\documentclass{amsart}
\usepackage[margin=1.2in]{geometry}
\usepackage{amssymb}
\usepackage{amscd}
\usepackage[all]{xy}
\usepackage{bbm}
\usepackage{mathrsfs}
\usepackage{enumerate}
\usepackage{tikz}
\usepackage{stmaryrd}
\usepackage{etoolbox}
\usepackage{array}
\usepackage{verbatim}
\usepackage{hyperref}
\usepackage{stmaryrd}
\usepackage{bbding}

\newcommand{\Z}{\mathbb{Z}}

\newcommand{\R}{\mathbb{R}}

\newcommand{\eps}{\varepsilon}

\newcommand{\del}{\nabla}

\newcommand{\lap}{\Delta}

\newcommand{\bd}{\partial}
\newcommand{\cl}{\overline}

\newcommand{\tr}{\operatorname{tr}}

\newcommand{\la}{\langle}
\newcommand{\ra}{\rangle}

\renewcommand{\div}{\operatorname{div}}
\newcommand{\weak}{\rightharpoonup}

\newcommand{\grad}{\del}

\newcommand{\an}{\operatorname{An}}

\theoremstyle{plain}
\newtheorem{theorem}{Theorem}
 
\newtheorem{corollary}[theorem]{Corollary}
\newtheorem{prop}[theorem]{Proposition}
\newtheorem{lemma}[theorem]{Lemma}
\newtheorem{conj}[theorem]{Conjecture}

\theoremstyle{definition}
\newtheorem{defn}[theorem]{Definition}
\newtheorem{rem}[theorem]{Remark}

\author{Liam Mazurowski}
\address{Department of Mathematics, Lehigh University, Bethlehem, Pennsylvania, 18015, United States}
\email{lim624@lehigh.edu}

\author{Xuan Yao}
\address{Department of Mathematics, Princeton University, Princeton, NJ 08540}
\email{xy1216@princeton.edu}

\title{A Positive Mass Theorem for Continuous Metrics}

\begin{document}

\begin{abstract}
    Let $g$ be a continuous metric on $\R^3$ which is asymptotically flat in the sense that $\vert g_{ij}(x) - \delta_{ij}\vert = O(\vert x\vert^{-\tau})$ for some $\tau > \frac 1 2$. Further assume that $g$ can be uniformly approximated on compact sets by smooth metrics with almost non-negative scalar curvature. For such a metric $g$, we define a synthetic ADM mass $m(g)$ using harmonic functions. The harmonic mass $m(g)$ coincides with the usual ADM mass whenever $g$ is smooth and decays rapidly enough that the latter is defined. The harmonic mass can also be computed as a limit of the $C^0$ local mass introduced by Burkhardt-Guim.  Our main result is a positive mass theorem: the harmonic mass satisfies $m(g)\ge 0$ and if $m(g) = 0$ then $g$ is flat.  
\end{abstract}

\maketitle

\section{Introduction} 
Although scalar curvature is defined using two derivatives of the metric, there is strong evidence to suggest that non-negativity of the scalar curvature is essentially a $C^0$ property.  For example, Gromov \cite{gromov2014dirac,gromov2019four} showed that non-negative scalar curvature is preserved under $C^0$ convergence of smooth manifolds to a smooth limit.  Gromov's proof relies on the non-existence of mean convex cubes with acute dihedral angles in manifolds with non-negative scalar curvature.  In fact, angles and mean convexity can be defined for $C^0$ metrics, and the associated prism rigidity theorem \cite{brendle2024scalar,li2020polyhedron,wang2021gromov} may be taken as a possible definition of non-negative scalar curvature for $C^0$ metrics.  Bamler \cite{bamler2016ricci} gave another proof of Gromov's $C^0$-convergence theorem using Ricci flow, and  Burkhardt-Guim \cite{burkhardt2019pointwise} extended these ideas to define a synthetic notion of  non-negative scalar curvature for $C^0$ metrics based on Ricci flow. Also see the  work of Lee \cite{lee2026quantification} building on Bamler's approach. We  refer to \cite{mazurowski2026quantification} and \cite{mazurowski2026scalar} for recent extensions of Gromov's $C^0$-convergence theorem proven using harmonic functions and $\mu$-bubbles, respectively. 

The positive mass theorem \cite{schoen1979proof,schoen1981proof,witten1981new} asserts that a smooth asymptotically flat manifold $(M,g)$ with non-negative scalar curvature must have non-negative ADM mass. Moreover, there is an associated rigidity statement: if the mass is zero then $M$ is flat.  The asymptotically flat condition in this theorem requires that $g$ converges suitably quickly in $C^1$ to the Euclidean metric at infinity, and indeed this type of decay is required to even define the classical ADM mass \cite{bartnik1986mass}.  In light of the above discussion, it is natural to wonder whether there is an analog of the positive mass theorem for continuous metrics. 

In this paper, we consider continuous metrics $g$ on $\R^3$ which are asymptotically flat in the sense that $\vert g_{ij}(x) - \delta_{ij}\vert = O(\vert x\vert^{-\tau})$ for some $\tau > \frac 1 2$. We further require that $g$ has non-negative scalar curvature in the sense of approximations. This means that $g$ can be uniformly approximated on compact sets by smooth metrics with almost non-negative scalar curvature. For any such metric $g$, we define a synthetic ADM mass $m(g)$ using harmonic functions. See Definition \ref{defn-harmonic-mass-I} for the precise formulation.    Our main result is the following positive mass theorem: 

\begin{theorem}
\label{main-theorem}
Assume that $g$ is a continuous asymptotically flat metric on $\R^3$ with non-negative scalar curvature in the sense of approximations. Then the harmonic mass $m(g)$ is a well-defined number in $[0,\infty]$. Moreover, if $m(g) = 0$ then $g$ is flat. 
\end{theorem}

\begin{rem}
Since $g$ is only continuous, one needs to be slightly careful with the definition of flatness.  Here we mean that there is a bi-Lipschitz homeomorphism $W\colon\R^3\to \R^3$ such that $W \colon (\R^3,d_{g}) \to (\R^3,d_{\text{euc}})$ is an isometry of metric spaces. Examples show that this bi-Lipschitz regularity for $W$ is optimal \cite{calabi1970smoothness}.
\end{rem}

The harmonic mass $m(g)$ coincides with the usual ADM mass whenever $g$ is smooth and decays rapidly enough that the latter is well-defined.  We can also show that the harmonic mass is computable as a limit of the $C^0$ local mass introduced by Burkhardt-Guim \cite{burkhardt2024adm}. In particular, there is a formula for the harmonic mass purely in terms of the metric coefficients near infinity. 

\begin{theorem} \label{theorem-coordinates} Let $g$ be a continuous asymptotically flat metric on $\R^3$ with non-negative scalar curvature in the sense of approximations. 
Let $\an(r) = B(0,4r)-B(0,r)$. There is a function $\eta  \in C^\infty_c((1,4))$  (that does not depend on $g$) such that the harmonic mass is given by
    \begin{align*}
        m(g) &= c \lim_{r\to \infty}\bigg[ \frac 1 r \int_{\an(r)} \left(\frac{1}{\vert x\vert} \eta\left(\frac{\vert x\vert}{r}\right) + \frac 1 r \eta'\left(\frac{\vert x\vert}{r}\right) \right)  \delta^{ij}(g_{ij}-\delta_{ij}) \,dx  \\
&\qquad \qquad  + \frac 1 r \int_{\an(r)} \left(\frac{1}{\vert x\vert} \eta\left(\frac{\vert x\vert}{r}\right) - \frac 1 r \eta'\left(\frac{\vert x\vert}{r}\right)\right) (g_{ij}-\delta_{ij}) \frac{x^ix^j}{\vert x\vert^2}\,dx\bigg].
\end{align*} 
Here $c = c(\eta) > 0$ is a fixed normalization constant. 
\end{theorem}

\begin{rem} 
Up to  normalization and shifting the annulus, the quantity inside the brackets on the right hand side above is exactly the $C^0$ local mass $M_{C^0}(g,\text{Id},\eta,r)$ defined by Burkhardt-Guim \cite[Definition 2.1]{burkhardt2024adm}. However, it is not quite clear whether the harmonic mass coincides with Burkhardt-Guim's global mass which is defined to be 
\[
\lim_{r\to\infty} M_{C^0}(g,\text{Id},\eta^r,r)
\]
for a certain family of functions $\eta^r$ converging smoothly to $\eta$ as $r\to \infty$. In Appendix \ref{Appendix-eta} we calculate an explicit formula for $\eta$, and we  show that 
\[
\lim_{r\to \infty} M_{C^0}(g,\text{Id},\eta^r,r) = \lim_{r\to\infty} M_{C^0}(g,\text{Id},\eta,r)
\]
when the decay rate satisfies $\tau > \frac 2 3$. This answers \cite[Question 2]{burkhardt2024adm} for metrics $g$ on $\R^3$ satisfying the faster decay rate $\tau > \frac 2 3$. 
\end{rem}

\begin{rem}
    The harmonic mass is manifestly independent of any choice of asymptotically flat coordinate system since it can be defined purely in terms of a harmonic function $u$ which is intrinsic to the $g$ metric. 
\end{rem}

\begin{rem} There is no conceptual obstruction to extending the results of this paper to complete manifolds $M^3$ satisfying $H_2(M;\Z)=0$, or to asymptotically flat manifolds with boundary satisfying $H_2(M,\bd M; \Z) = 0$.  We plan to address this in future work. 
\end{rem} 

As a further application of these results, we can resolve a conjecture of Gromov \cite[Section 3.11]{gromov2019four} on metrics with non-negative scalar curvature that decay rapidly at infinity. 

\begin{conj}[Euclidean $C^0$-Rigidity Conjecture (a)] 
\label{gromov-conj} Assume that $g$ is a smooth metric on $\R^3$ with non-negative scalar curvature. If $g$ satisfies $\vert g_{ij}(x)-\delta_{ij} \vert = o(\vert x\vert^{-1})$ then $g$ is flat. 
\end{conj}

Indeed,  if $g$ satisfies $\vert g_{ij}(x) - \delta_{ij}\vert = o(\vert x\vert^{-1})$ then the formula in Theorem \ref{theorem-coordinates} immediately implies that $m(g) = 0$. Conjecture \ref{gromov-conj} then follows from the rigidity statement in Theorem \ref{main-theorem}. In fact, we can even obtain the conclusion when $g$ is merely continuous rather than smooth. 

\begin{corollary}
\label{gromov-corollary}
    Assume that $g$ is a continuous metric on $\R^3$ with non-negative scalar curvature in the sense of approximations. If $g$ satisfies $\vert g_{ij}(x) - \delta_{ij}\vert = o(\vert x\vert^{-1})$ then $g$ is flat. 
\end{corollary}

\begin{rem}
    We would like to note that Conjecture \ref{gromov-conj} (for smooth metrics) can also be deduced as a  corollary of the very nice isoperimetric mass rigidity theorem of Benatti-Fogagnolo-Mazzieri  \cite[Theorem 2.14]{benatti2025isoperimetric}. See \cite{you2026gromov} for another proof of Conjecture \ref{gromov-conj} (for smooth metrics) by extending the results in an earlier version of this work. 
\end{rem}

\subsection{Discussion} There is great interest in extending the positive mass theorem below the classical $C^2$ regularity threshold. Many different approaches to this problem have been considered.   One approach studies Sobolev metrics with non-negative scalar curvature in the sense of distributions. In this direction, Lee and LeFloch \cite{lee2015positive} proved a positive mass theorem with rigidity for $C^0 \cap W^{1,n}$ metrics using techniques from spin geometry. 

Another approach is based on Huisken's isoperimetric mass \cite{huisken2006isoperimetric}. The isoperimetric mass is appealing because it is easy to define on manifolds with only $C^0$ regularity.  However, it was observed by Jauregui, Lee, and Unger \cite{jauregui2024note} that the isoperimetric mass is always non-negative for elementary reasons, and that this does not require non-negativity of the scalar curvature.  Nevertheless, there is still an interesting rigidity question: if a manifold has zero isoperimetric mass does this imply that it is flat? For smooth metrics with $C^0$ decay, this isoperimetric rigidity question was  resolved in the affirmative by Bennati-Fogagnolo-Mazzieri \cite{benatti2025isoperimetric}. Their approach is based on the weak inverse mean curvature flow. In this direction, we note that Fogagnolo-Gatti-Pluda \cite{fogagnolo2026scalar} have just announced work extending the weak inverse mean curvature flow to continuous metrics.
Finally, we mention the related isocapacitary mass introduced by Jauregui \cite{jauregui2024adm}, and note that it also has applications to the low regularity setting \cite{benatti2023nonlinear}. 

A third approach is based on the Ricci flow. Burkhardt-Guim \cite{burkhardt2019pointwise} defined a synthetic notion of scalar curvature lower bounds for $C^0$ metrics using the Ricci flow.  Then in \cite{burkhardt2024adm}, Burkhardt-Guim defined a mass for $C^0$ asymptotically flat manifolds with non-negative scalar curvature in the Ricci flow sense, and showed that this mass is a well-defined geometric quantity.  It is still an open problem whether this mass must be non-negative in general \cite[Question 2]{burkhardt2024adm}.  

In this paper, we consider continuous metrics which have non-negative scalar curvature in the sense of approximations. Again, this means that the metric  can be uniformly approximated on compact sets by smooth metrics with almost non-negative scalar curvature; see Definition \ref{defn-non-negative-scalar}. This is a relatively general notion of non-negative scalar curvature for continuous metrics. For example, any continuous metric $g$ with non-negative scalar curvature in the Ricci flow sense also satisfies this approximation property \cite[Lemma 7.2]{burkhardt2019pointwise}.  To the authors' knowledge, Theorem \ref{main-theorem} is the first positive mass theorem with rigidity for continuous metrics.

\subsection{Sketch of Proof} The ideas in this paper are based on the harmonic function method for studying the  scalar curvature of metrics with $C^0$ control developed by the authors in \cite{mazurowski2026quantification}.  Assume that $g$ is a continuous asymptotically flat metric on $\R^3$.   It is well-known that such a metric $g$ admits a Green's function $u$ for $\lap_g$ with a pole at any fixed point $x_0$. We study a quantity $D(r)$ associated to $u$ that captures information about the harmonic function $u$ and its gradient on the set $\{1/(4r) < u < 1/r\}$. This quantity $D(r)$ depends continuously on $u$ in $W^{1,p}$ for $p > 3$ and is therefore stable under small $C^0$ changes to the metric. When $g$ is smooth, $D(r)$ can be re-written as an integral of the $F$ function introduced in \cite{agostiniani2024green}. By the monotonicity formula for the $F$ function \cite{agostiniani2024green}, it follows that $r\mapsto r D(r)$ is almost non-decreasing when $g$ is smooth with almost non-negative scalar curvature. Passing to a limit, it follows that $r\mapsto r D(r)$ is non-decreasing whenever $g$ is continuous with non-negative scalar curvature in the sense of approximations. 

Now suppose $g$ is a continuous asymptotically flat metric on $\R^3$ with non-negative scalar curvature in the sense of approximations. By the above monotonicity, there is a well-defined limit 
\[
m(g) = c \lim_{r\to \infty}rD(r) \in (-\infty,\infty]
\]
which we call the {\it harmonic mass} of $g$. Here $c$ is a normalization constant chosen so that $m(g)$ coincides with the usual ADM mass when $g$ is smooth and decays suitably rapidly at infinity. By linearizing $D$ near infinity, we can show that $m(g)$ can be expressed purely in terms of the metric coefficients $g_{ij}$ near infinity. In particular, this implies that the harmonic mass $m(g)$ does not depend on any choices in the above construction. To show that $m(g)\ge 0$, we prove an asymptotic expansion for the Green's function $u$ near the pole $x_0$. More precisely, we show that (after a linear change of coordinates) the rescaled functions $r u(r(x-x_0))$ converge to $\vert x\vert^{-1}$ in $W^{1,p}$ on a fixed annulus as $r\to 0$. This is sufficient to show that $D(r) \to 0$ as $r\to 0$ and therefore that $rD(r) \ge 0$ for all $r$ by monotonicity. 

It remains to prove the rigidity when $m(g) = 0$. This is the most delicate part of the argument. By monotonicity, the equality $m(g) = 0$ implies that $D(r) = 0$ for all $r$. On a fixed compact annulus $\an$ around the pole, we choose smooth approximating metrics $g_k$ and approximating harmonic functions $u_k$. Then $D_k(r)\to D(r)$ as $k\to \infty$ and it follows that  $D_k(r_2) - D_k(r_1) \to 0$ as $k\to \infty$ for any fixed $r_1< r_2$. By the monotonicity formula for $F_k$ in the smooth case, this implies that 
\begin{equation}
\label{stable-eqn}
\int_{a}^{b} \int_{\{u_k=1/t\}} \frac{\vert \grad^{g_k,T} \vert \grad^{g_k} u_k\vert \vert^2}{\vert \grad^{g_k}u_k\vert^2} + \left(H_k - \frac{2\vert \grad^{g_k} u_k\vert}{u_k}\right)^2\, da_k\, dt \to 0
\end{equation}
where $\grad^{g_k,T}$ is the tangential gradient on $\{u_k=1/t\}$ and $H_k$ is the mean curvature of $\{u_k=1/t\}$. Next, observe that the metrics $g_k$ have uniform Poincaré inequalities since $g_k \to g$ uniformly. Hence there are constants $c_k$ such that 
\[
\int_{\an} \left\vert \frac{\vert \grad^{g_k} u_k\vert}{u_k^2} - c_k\right\vert\, dv_k \le C \int_{\an} \left\vert \grad^{g_k}\left(\frac{\vert \grad^{g_k} u_k\vert}{u_k^2}\right)\right\vert\, dv_k,
\]
where, crucially, $C$ does not depend on $k$. Estimate \eqref{stable-eqn} implies that the right hand side goes to 0 as $k\to \infty$, and so we can pass to a limit to deduce that 
\[
\vert \grad^g u\vert = cu^2
\]
almost-everywhere. Then using $D(r) = 0$, one can check that necessarily $c = 1$. 

Next, consider $\rho = u^{-1}$. The above calculations imply that $\vert \grad^g \rho\vert = 1$ almost-everywhere and that $\lap_g(\rho^2) = 6$ in the weak sense everywhere on $\R^3$ (even across the pole). We then vary the choice of the pole. Write $\rho_y$ for the $\rho$ function associated to a Green's function with pole at $y$. Then for any $y_1$ and $y_2$, we have 
\[
\lap_g(\rho_{y_1}^2 - \rho_{y_2}^2) = 0
\]
in the weak sense on all of $\R^3$. Moreover,  $\rho_{y_1}^2 - \rho_{y_2}^2$ has sub-quadratic growth at infinity.  Since $g$ is asymptotically flat, we can show that the space of $g$-harmonic functions with  sub-quadratic growth at infinity has a basis $\{1,w_1,w_2,w_3\}$ where $w_i$ is a $g$-harmonic function asymptotic to the coordinate function $x^{i}$. Proving this requires a thorough understanding of the analytic properties of the equation $\lap w = \div(X)$ on $\R^3$. In the literature, this equation is usually studied in conjunction with the Helmholtz decomposition; see for example \cite{farwig1997weighted}.  

Once this is established, we have 
\[
\rho_{y_1}^2 - \rho_{y_2}^2 = b_0 + b_1 w_1 + b_2 w_2 + b_3 w_3
\]
for some constants $b_i$ (depending on $y_1$ and $y_2$). Define a Lipschitz map $W\colon \R^3\to \R^3$ by $W(x) = (w_1(x),w_2(x),w_3(x))$. Using the fact that the above identity holds for all $y_1$ and $y_2$, one can check that necessarily 
\begin{gather*}
    g(\grad^g w_i,\grad^g w_j) = \delta_{ij},\quad 
    \rho_y(x) = \vert W(x)-W(y)\vert,
\end{gather*}
almost-everywhere. Finally, we use these properties to show that $W$ is a homeomorphism, that the inverse $W^{-1}$ is also Lipschitz, and finally that $W\colon(\R^3,d_g) \to (\R^3,d_{\text{euc}})$ is an isometry of metric spaces.

\subsection{Organization} The remainder of the paper is organized as follows. In Section \ref{section-prelim}, we define asymptotically flat metrics with non-negative scalar curvature in the sense of approximations. Then we discuss some background on Green's functions for such metrics and define the quantity $D(r)$ and the harmonic mass $m(g)$.  In Section \ref{section-monotonicity}, we prove that $r\mapsto rD(r)$ is non-decreasing when $g$ has non-negative scalar curvature in the sense of approximations. In Section \ref{section-pos}, we show the harmonic mass is non-negative. In Section \ref{section-choices}, we prove Theorem \ref{theorem-coordinates} and deduce that the harmonic mass is independent of any choices used in the construction. Finally, in Section \ref{section-rigidity} we complete the proof of Theorem \ref{main-theorem} by demonstrating the rigidity of metrics satisfying $m(g) = 0$. We then give the application to Gromov's Euclidean $C^0$-Rigidity Conjecture (Corollary \ref{gromov-corollary}).  The paper concludes with four appendices. In Appendix \ref{D-appendix}, we derive the linearization of $\mathcal D$ near infinity which is needed in the proof of Theorem \ref{theorem-coordinates}.  In Appendix \ref{Appendix-F}, we recall the basic properties of the $F$ function introduced in \cite{agostiniani2024green}. In Appendix \ref{Appendix-Mapping}, we discuss the analytic properties of the equation $\lap w = \div(X)$ on $\R^3$.  Finally, in Appendix \ref{Appendix-eta} we explore the relation between the harmonic mass and Burkhardt-Guim's global mass. 

\subsection{Acknowledgments}
We would like to thank Xin Zhou for his interest in this work. We are grateful to Mattia Fogagnolo for bringing several helpful references to our attention after the first version of this work appeared. 

\section{Preliminaries} 
\label{section-prelim}

In this section, we define continuous asymptotically flat metrics as well as the notion of non-negative scalar curvature in the sense of approximations. Then we discuss some preliminary results on the Laplace Green's function for such metrics. Finally, we give the definition of the harmonic mass $m(g)$.

\begin{defn} 
Let $g$ be a continuous metric on $\R^3$. We say that $g$ is {\it asymptotically flat} if 
\[
\vert g_{ij}(x) - \delta_{ij}\vert = O(\vert x\vert^{-\tau})
\]
for some $\tau > \frac 1 2$. 
\end{defn} 

\begin{defn}
\label{defn-non-negative-scalar}
Let $g$ be a continuous asymptotically flat metric on $\R^3$.  We say that $g$ has {\it non-negative scalar curvature in the sense of approximations} if for each $r > 0$ there is a sequence of smooth metrics $g_k$ on $B(0,r)$ such that 
\begin{itemize}
    \item[(i)] $g_k \to g$ uniformly on $B(0,r)$, and 
    \item[(ii)] $R(g_k)\ge - \eps_k$ on $B(0,r)$
\end{itemize}
for some numbers $\eps_k\to 0$. 
\end{defn}

\begin{rem}
    We have chosen to use a pointwise lower bound in (ii) for simplicity, but this condition is not sharp. For example, one can check that the arguments of the paper still hold assuming only an $L^p$ bound 
    \[
    \int_{B(0,r)} \vert R(g_k)_-\vert ^p \, dv_{g_k} \le \eps_k 
    \]
    for some $p > 1$. 
    Here $R(g_k)_-$ is the negative part of $R(g_k)$. 
\end{rem}

\subsection{Green's Functions}
Let $g$ be a continuous asymptotically flat metric on $\R^3$.  Let $u$ be a Green's function for $\lap_g$ with pole at some point $x_0\in \R^3$. It is well-known that such a function exists and satisfies two-sided bounds 
\[
\frac{C_1}{\vert x-x_0\vert} \le u(x) \le \frac{C_2}{\vert x-x_0\vert}
\]
for some constants $0 < C_1 < C_2$; see  for example \cite[Theorem (1.1)]{gruter1982green}. Moreover, $u$ belongs to $W^{1,p}_{\text{loc}}(\R^3 - \{x_0\})$ for all $1\le p < \infty$.  In particular, $u$ is continuous. By the co-area formula for Sobolev functions, the flux 
\[
\int_{\{u = 1/t\}} \vert \grad^g u\vert\, d\mathcal H^2_g
\]
is well-defined for almost every $t$. Moreover, this quantity is independent of $t$, and we call it the flux of $\grad u$ over a level set. We will always assume $u$ is normalized so that the flux of $\grad u$ over a level set is $4\pi$.

\subsection{The Harmonic Mass} Assume that $g$ is a continuous asymptotically flat metric on $\R^3$ with non-negative scalar curvature in the sense of approximations. 
Let $\psi$ be a smooth, non-negative bump function which is compactly supported in $(0,1)$. We will always assume that $\psi$ is not identically zero.  
Then define 
\begin{equation}
\label{eqn-theta}
\theta(r,t) = \frac{1}{t^3} \left[\frac 1 2 \psi\left(\frac{1}{2r t} -1\right) - \psi\left(\frac{1}{rt}-1\right)\right]. 
\end{equation}
We now define the $D$ function associated $u$ and $\psi$.  

\begin{defn} \label{defn-D} Let $u$ be a Green's function for $\lap_g$ with pole at $x_0$.  
For $r > 0$, define   $D(r)$ by 
\[
D(r) = c_\psi + \int \theta(r,u) \vert \grad^g u\vert^3\, dv_g
\]
where 
\[
c_\psi = 2\pi \int_0^1 \frac{\psi(s)}{1+s}\, ds
\]
is a constant that depends only on the choice of $\psi$.  
\end{defn} 

We note that $D(r)$ of course depends on both $u$ and $\psi$ but we will suppress this in the notation except where needed for clarity.  We also note that $\theta(r,u)$ vanishes outside the set $\{1/(4r) < u < 1/r\}$ since $\psi$ is supported in $(0,1)$.
In Section \ref{section-monotonicity}, we show that $r\mapsto rD(r)$ is non-decreasing.  This allows us to use $D(r)$ to define the harmonic mass. 

\begin{defn}
\label{defn-harmonic-mass-I}
Assume that $g$ is a continuous asymptotically flat metric on $\R^3$ which has non-negative scalar curvature in the sense of approximations. Let $u$ be a Green's function for $\lap_g$ and let $\psi$ be a smooth non-negative bump function supported in $(0,1)$.  Define the harmonic mass 
\[
m(g,u,\psi) = \left[3\pi \int_0^1 \frac{\psi(s)}{(1+s)^2}\, ds\right]^{-1} \lim_{r\to \infty} rD(r) \in (-\infty,\infty]
\]
where $D(r)$ is the function in Definition \ref{defn-D}.
\end{defn}

A priori, the harmonic mass $m(g,u,\psi)$ may depend on both the choice of $u$ and $\psi$.  In Section \ref{section-choices}, we will show that the value $m(g,u,\psi)$ is in fact independent of both $u$ and $\psi$. Thus we can make the following definition. 

\begin{defn}
    Assume that $g$ is a continuous asymptotically flat metric on $\R^3$ which has non-negative scalar curvature in the sense of approximations. The harmonic mass $m(g)$ is defined to be $m(g,u,\psi)$ for any choice of Green's function $u$ and bump function $\psi$. 
\end{defn}

\section{Monotonicity} \label{section-monotonicity}

Let $g$ be a continuous asymptotically flat metric on $\R^3$ with non-negative scalar curvature in the sense of approximations. Let $u$ be a Green's function for $\lap_g$ with pole at $x_0$. Fix a smooth, non-negative bump function $\psi$ which is supported in $(0,1)$ and consider the quantity $D(r)$ associated to $u$ and $\psi$.  The goal of this section is to prove the following monotonicity theorem. 

\begin{theorem}
\label{theorem-monotonicity} 
The function $r\mapsto rD(r)$ is non-decreasing. 
\end{theorem} 

This has the following immediate corollary. 

\begin{corollary}
\label{corollary-well-defined}
The limit 
\[
m(g,u,\psi) := \lim_{r\to \infty} \left[3\pi \int_0^1 \frac{\psi(s)}{(1+s)^2}\right]^{-1} r D(r)
\]
exists and belongs to $(-\infty,\infty]$. 
\end{corollary}

We now turn our attention to proving Theorem \ref{theorem-monotonicity}. Fix a number $a > 0$.  We can find an open Euclidean annulus $\an := B(x_0,r_2)- \overline B(x_0,r_1)$ so that $\{1/(8a) < u < 1/a\}$ is compactly contained in $\an$.  Since $g$ has non-negative scalar curvature in the sense of approximations, we can find smooth metrics $g_k$ with scalar curvature at least $-\eps_k$ converging uniformly to $g$ on $\an$.  Define $u_k$ to be the $g_k$-harmonic function on $\an$ which agrees with $u$ on the boundary of $\an$. 

\begin{prop}
\label{W1pConvergence}
The functions $u_k$ converge to $u$ in $W^{1,p}_{\operatorname{loc}}(\an)$ for every $1 < p < \infty$. 
\end{prop}

\begin{proof}
First we show that $u_k \to u$ in $W^{1,2}(\an)$.  Define $A_k = (g_k)^{ij} \sqrt{\det g_k}$ and $A = g^{ij}\sqrt{\det g}$. Then 
\[
\div(A_k(x) \grad u_k) = 0 \quad \text{and}\quad \div(A(x) \grad u) = 0
\]
in the weak sense. We test these equations against $u_k - u$ and use the fact that $u_k-u$ vanishes on the boundary of $\an$ to get 
\begin{gather*}
\int_{\an} \la A_k(x)\grad u_k, \grad u_k - \grad u\ra \, dx = 0,\\
\int_{\an} \la A(x) \grad u, \grad u_k - \grad u\ra\, dx = 0. 
\end{gather*}
This implies that 
\[
0 = \int_{\an} \la (A_k(x) - A(x))\grad u_k, \grad u_k-\grad u\ra\, dx + \int_{\an} \la A(x)(\grad u_k - \grad u), \grad u_k-\grad u\ra\, dx. 
\]
By the uniform ellipticity of $A$, we deduce that 
\[
\int_{\an} \vert \grad u_k - \grad u\vert^2\, dx \le C \|A_k - A\|_{L^\infty} \int_{\an} \vert \grad u_k\vert \vert \grad u_k - \grad u\vert\, dx. 
\]
By the Cauchy inequality with $\eps$, this implies that 
\[
\int_{\an} \vert \grad u_k-\grad u\vert^2\, dx \le C \|A_k - A\|^2_{L^\infty} \int_{\an} \vert \grad u_k\vert^2\, dx. 
\]
Finally, since $u_k$ minimizes the $g_k$-Dirichlet energy for the given boundary values and $g_k \to g$ uniformly, we have 
\[
\int_{\an} \vert \grad u_k\vert^2\, dx \le C. 
\]
It follows that 
\[
\int_{\an} \vert \grad u_k -\grad u\vert^2\, dx \to 0 
\]
as $k\to \infty$.  Since $u_k = u$ on the boundary of $\an$, it follows from the Poincaré inequality that $u_k \to u$ in $W^{1,2}(\an)$. 

Next we upgrade this to convergence in $W^{1,p}_{\text{loc}}(\an)$. Fix some $3 < p < q < \infty$.  Fix a point $x\in \an$. By a linear change of coordinates, we can suppose that $g(x) = I$. Then $A_k$ will be close to $I$ in a small ball $2B$ centered at $x$ for large $k$. This allows one to apply Meyer's $L^q$ gradient estimate \cite[Theorem 2]{meyers1963p} to deduce that 
\[
\left[\int_B \vert \grad u_k\vert^q\, dx\right]^{1/q} \le C \left[\int_{2B} \vert u_k\vert^2\, dx\right]^{1/2} \le C. 
\]
Changing back to the original coordinates and using a covering argument, this implies that 
\[
\int_K \vert \grad u_k\vert^q\, dx \le C(K)
\]
for every compact subset $K$ of $\an$. Then since $\grad u_k \to \grad u$ in $L^2$ and $\grad u_k$ and $\grad u$ are uniformly bounded in $L^q$, it follows from interpolation that $\grad u_k \to \grad u$ in $L^p$ since $p < q$. Finally, again by the Poincaré inequality it follows that $u_k \to u$ in $W^{1,p}(K)$ for every compact set $K\subset \an$. 
\end{proof}

Let $D_k(r)$ be the $D$ quantity associated to $u_k$ and $\psi$. More precisely, define 
\[
D_k(r) = c_\psi + \int \theta(r,u_k) \vert \grad^{g_k} u_k\vert^3 \, dv_{g_k}
\]
where $\theta$ is defined by \eqref{eqn-theta}.
Since $\psi$ is supported in $(0,1)$ and $u_k \to u$ in $C^0$, this is well-defined for all $r\in [a,2a]$ once $k$ is large enough. Since the metric $g_k$ is smooth, we can consider the $F_k$ function associated to $u_k$: 
\[
F_k(t) = 4\pi t - t^2 \int_{\{u_k=\frac 1 t\}} H_k \vert \grad^{g_k} u_k\vert\, da_k + t^3 \int_{\{u_k=\frac 1 t\}} \vert \grad^{g_k} u_k\vert^2\, da_k.
\]
See Appendix \ref{Appendix-F} for the basic properties of the $F$ function. 
Again since $g_k$ is smooth, we can re-write $r D_k(r)$ in terms of the $F_k$ function. 

\begin{prop} \label{DvsF} We have 
\[
rD_k(r) = \int_0^1 \frac{\psi(v)}{(1+v)^2} \left[\int_1^2 \frac{F_k((1+v)r w)}{w^3}\, dw\right]\, dv. 
\]
\end{prop} 

\begin{proof}
Using Proposition \ref{ac-gradient} and the co-area formula, one computes that 
\[
\int_0^r r \psi(s/r)\left[ \int_{r+s}^{2r+2s} \frac{F_k(t)}{t^3}\,dt\right]\, ds = r D_k(r). 
\]
The formula now follows by changing variables in the integrals.
\end{proof}

Crucially, the functions $D_k(r)
$ satisfy the following almost-monotonicity property. 

\begin{prop}
For any $a \le r_1 < r_2 \le 2a$, we have 
\[
r_2 D_k(r_2) - r_1 D_k(r_1) \ge - C \eps_k,
\]
where $C$ is a constant that does not depend on $k$. 
\end{prop}

\begin{proof}
We continue to work with $\an$ defined as below Corollary \ref{corollary-well-defined}. First, we claim that the level set $\{u_k = 1/t\}$ is connected for every $t\in [a,8a]$ such that $1/t$ is a regular value of $u_k$.  Note that the $C^0$ convergence $u_k\to u$ implies that such a level set is compactly contained in the interior of $\an$ for large $k$. Suppose for contradiction that $\{u_k = 1/t\}$ is not connected.  Then it contains two distinct components $\Sigma_1$ and $\Sigma_2$, each of which is a smooth, closed surface contained in the interior of $\an$. Since we are in $\R^3$, these enclose unique bounded open regions $\Omega_1$ and $\Omega_2$. Neither of these can be contained in $\an$ by the maximum principle. The only other possibility is that $\Omega_1$ and $\Omega_2$ are nested, but this also violates the maximum principle. 

The formula for $F_k'(t)$ in Proposition \ref{F-derivative} implies that 
\[
F_k'(t) \ge \int_{\{u_k=1/t\}} \frac{R_{g_k}}{2}\, da_k.
\]
Since the level sets of $u_k$  are connected, one can integrate $F_k'$ and apply the co-area formula to deduce the following almost-monotonicity formula for $F_k$: 
\[
F_k(t_2) - F_k(t_1) \ge \int_{\{1/t_2 < u_k < 1/t_1\}} R_{g_k} \vert \grad^{g_k} u_k\vert\, dv_k \ge -C \eps_k, \quad a \le t_1 < t_2 \le 8a.
\]
Here $C$ is a constant that does not depend on $k$ since $\an$ is fixed and the $g_k$-Dirichlet energy of $u_k$ is uniformly controlled. By Proposition \ref{DvsF}, this implies 
\[
r_2 D_k(r_2) - r_1D_k(r_1) \ge - C\eps_k 
\]
for all $a\le r_1 < r_2 \le 2a$, as needed. 
\end{proof} 

Finally, the $W^{1,p}$ convergence of $u_k$ to $u$ implies that $D_k(r) \to D(r)$ as $k\to \infty$ for every $r\in [a,2a]$. Passing to the limit in the almost-monotonicity formula, it follows that $r \mapsto r D(r)$ is non-decreasing on the interval $[a,2a]$.  Theorem \ref{theorem-monotonicity} now follows since $a$ was arbitrary. 

\section{Positivity of the Mass} 
\label{section-pos} 

Let $g$ be a continuous asymptotically flat metric on $\R^3$ with non-negative scalar curvature in the sense of approximations. Let $u$ be a Green's function for $\lap_g$ with pole at $x_0$, normalized so that the flux of $\grad u$ over any level set is $4\pi$.  The goal of this section is to establish the following theorem. 

\begin{theorem}
The harmonic mass $m(g,u,\psi)$ is non-negative. 
\end{theorem}

In light of the monotonicity formula in Theorem \ref{theorem-monotonicity} and the definition of $m(g,u,\psi)$, this will follow immediately from the following proposition. 

\begin{prop}
We have $\lim_{r\to 0} r D(r) = 0$. 
\end{prop}

\begin{proof}
The main step is to establish a suitable asymptotic expansion for $u$ near the pole. By an initial affine change of coordinates, we can assume without loss of generality that the pole is at the origin and $g(0) = I$.  Define $A(x) = g^{ij} \sqrt{\det g}$ so that 
\[
\div(A(x) \grad u) = 0
\]
away from the pole. 
Define rescaled metrics $g_r(x) = g(rx)$, and rescaled coefficients $A_r(x) = A(rx)$, and rescaled functions $u_r(x) = r u(rx)$. The rescaled functions $u_r$ solve 
\[
\div(A_r(x) \grad u_r) = 0
\]
in $\R^3 - \{0\}$. Also note that 
\[
\frac{C_1}{\vert x\vert} \le u_r(x) \le \frac{C_2}{\vert x\vert}
\]
for some constants $0 < C_1 < C_2$. 

Fix $a_1 < a_2 < b_2 < b_1$ and consider the annuli $\an_2 = B(0,b_2) - B(0,a_2) \subset \an_1 = B(0,b_1) - B(0,a_1)$.  Since $g$ is continuous and $g(0) = I$, we know that $A_r(x) \to I$ uniformly on $\an_1$ as $r\to 0$.  Fix any $p > 3$. By Meyer's $L^p$ gradient estimate \cite[Theorem 2]{meyers1963p}, we have 
\[
\|\grad u_r\|_{L^p(\an_2)} \le C \|u_r\|_{L^2(\an_1)}.
\]
The right hand side is uniformly bounded independent of $r$, and hence we see that $u_r$ is uniformly bounded in $W^{1,p}(\an_2)$.  Passing to a subsequence and using a diagonal argument, we can therefore suppose that $u_{r_k} \weak w$ weakly in $W^{1,p}_{\text{loc}}(\R^3 - \{0\})$ and $u_{r_k} \to w$ in $C^0_{\text{loc}}(\R^3 - \{0\})$. Since $A_r \to I$, the weak limit $w$ is a weak solution to $\div(\grad w) = 0$ in $\R^3 - \{0\}$ and therefore $w$ is actually a smooth harmonic function. Moreover, $w$ satisfies 
\[
\frac{C_1}{\vert x\vert} \le w(x) \le \frac{C_2}{\vert x\vert}
\]
for all $x\in \R^3 - \{0\}$. By B\^ocher's theorem, there is a constant $a > 0$ and a harmonic function $v(x)$ on $\R^3$ such that 
\[
w(x) = \frac{a}{\vert x\vert} + v(x). 
\]
Since $w(x) \to 0$ at infinity, the Liouville theorem implies that $v\equiv 0$ and therefore $w(x) = a \vert x\vert^{-1}$. 

Continuing to work with the same subsequence, we claim that $u_{r_k} \to w$ strongly in $W^{1,p}_{\text{loc}}(\R^3 - \{0\})$ for every $p > 3$.  Indeed, consider $\an_2 \subset \an_1$ as above. Write $A_r(x) = I - B_r(x)$ so that $B_r \to 0$ locally uniformly.  Let $w_{r_k} = u_{r_k} - w$.  Then we have 
\[
\div(A_{r_k}(x) \grad w_{r_k}) = \div(B_{r_k}(x)\grad w).
\]
Fix any $p > 3$. The coefficients $A_{r_k}(x)$ are converging to $I$, so Meyer's $L^p$ gradient estimate \cite[Theorem 2]{meyers1963p} implies that 
\[
\|\grad w_{r_k}\|_{L^p(\an_2)} \le C \|w_{r_k}\|_{L^2(\an_1)} + C \|B_{r_k}(x) \grad w\|_{L^p(\an_1)}. 
\]
The right hand side goes to 0 as $r_k\to 0$, and therefore we obtain $w_{r_k} \to 0$ in $W^{1,p}_{\text{loc}}(\R^3 - \{0\})$, as claimed.  We note that strong convergence in $W^{1,p}$ is enough to obtain the convergence of the fluxes and therefore $w = \vert x\vert^{-1}$. Since the limit is independent of the subsequence, we have $u_r \to \vert x\vert^{-1}$ in $W^{1,p}_{\text{loc}}(\R^3 - \{0\})$ as $r\to 0$.

Now observe that by a change of variables we have  $D(r) = \mathcal D(g_r,u_r)$ where 
\begin{align*}
\mathcal D(h,f) = c_\psi + \int_{\an} \varphi(f) \vert \grad^h f\vert^3\, dv_h,
\end{align*} 
and we have defined $\varphi(f) = \theta(1,f)$ for $\theta$ as in \eqref{eqn-theta} and $\an = B(0,4)-B(0,1)$.  The function $\mathcal D$ is a continuous function of the pair $(h,f)$ in $C^0 \times W^{1,p}$ for $p > 3$ in a neighborhood of $(g_{\text{euc}},\vert x\vert^{-1})$. Therefore we have $\mathcal D(g_r,u_r) \to \mathcal D(g_{\text{euc}}, \vert x\vert^{-1}) = 0$ as $r\to 0$.  Thus $D(r) \to 0$ as $r\to 0$ and so certainly $rD(r) \to 0$ as $r\to 0$ as well. 
\end{proof} 

\section{Independence of the Choices}

\label{section-choices}

In this section, we show that the harmonic mass $m(g,u,\psi)$ does not depend on the choice of the Green's function $u$ or the choice of the function $\psi$.   
First, we show that the harmonic mass does not depend on $\psi$ for fixed $u$. For convenience, define $Q_\psi(r) = rD_\psi(r)$.  

\begin{prop}
Fix smooth, non-negative bump functions $\psi_1$ and $\psi_2$ compactly supported in $(0,1)$ with 
\[
\int_0^1 \frac{\psi_1(s)}{(1+s)^2}\, ds = \int_0^1 \frac{\psi_2(s)}{(1+s)^2}\, ds = 1. 
\]
Then for any $r > 0$ we have $Q_{\psi_1}(r) \le Q_{\psi_2}(4r)$ and $Q_{\psi_2}(r) \le Q_{\psi_1}(4r)$. 
\end{prop}

\begin{proof} Let $\an$ be a Euclidean annulus which compactly contains $\{1/(16r) < u < 1/r\}$. Choose smooth metrics $g_k$ converging uniformly to $g$ on $\an$ with $R_{g_k} \ge -\eps_k$. Let $u_k$ be the $g_k$-harmonic function on $\an$ which equals $u$ on the boundary of $\an$.  As we have already seen in Proposition \ref{W1pConvergence}, the functions $u_k$ converge to $u$ in $W^{1,p}_{\text{loc}}(\an)$. Moreover, the $F_k$ function associated to $u_k$ satisfies 
\[
F_k(t_2) - F_k(t_1) \ge -C \eps_k 
\]
for all $r \le t_1 < t_2 \le 16r$, where $C$ is a positive constant independent of $k$.  Hence by Proposition \ref{DvsF}, we can estimate 
\begin{align*}
Q_{k,\psi_1}(r) &= \int_0^1 \frac{\psi_1(v)}{(1+v)^2}\left[\int_1^2 \frac{F_k((1+v)rw)}{w^3}\, dw\right] \,dv\\
&\le \int_0^1 \frac{\psi_1(v)}{(1+v)^2}\left[\int_1^2 \frac{F_k(4r) + C\eps_k}{w^3}\, dw\right] \,dv = \frac{3}{8} F_k(4r) + C\eps_k. 
\end{align*} 
Likewise we can estimate 
\begin{align*}
    Q_{k,\psi_2}(4r) &= \int_0^1 \frac{\psi_2(v)}{(1+v)^2}\left[\int_1^2 \frac{F_k((1+v)4rw)}{w^3}\, dw\right]\, dv\\
    &\ge \int_0^1 \frac{\psi_2(v)}{(1+v)^2}\left[\int_1^2 \frac{F_k(4r) - C\eps_k}{w^3}\, dw\right]\, dv = \frac 3 8 F_k(4r) - C\eps_k. 
\end{align*}
Since $Q_{k,\psi_1}(r) \to Q_{\psi_1}(r)$ and $Q_{k,\psi_2}(4r)\to Q_{\psi_2}(4r)$, we can pass to the limit in the above inequalities to deduce that 
\[
Q_{\psi_1}(r) \le Q_{\psi_2}(4r). 
\]
Interchanging the roles of $\psi_1$ and $\psi_2$ gives the opposite inequality, and the proposition follows.  
\end{proof}

This immediately implies that $m(g,u,\psi)$ does not depend on $\psi$. 

\begin{corollary}
For any admissible choices of $\psi_1$ and $\psi_2$ we have $m(g,u,\psi_1) = m(g,u,\psi_2)$. 
\end{corollary}

\begin{proof}
Note that replacing $\psi$ with a constant multiple of $\psi$ does not change $m(g,u,\psi)$. Therefore, we can assume without loss of generality that  $\psi_1$ and $\psi_2$ satisfy 
\[
\int_0^1 \frac{\psi_1(s)}{(1+s)^2}\, ds = \int_0^1 \frac{\psi_2(s)}{(1+s)^2}\, ds = 1. 
\]
We have already seen that 
\[
m(g,u,\psi_i) := \frac{1}{3\pi} \lim_{r\to \infty} Q_{\psi_i}(r)
\]
exists and belongs to $[0,\infty]$ for $i=1,2$. The previous proposition then implies that 
\[
m(g,u,\psi_1) \le  m(g,u,\psi_2)\le  m(g,u,\psi_1)
\]
and it follows that $m(g,u,\psi_1) = m(g,u,\psi_2)$. 
\end{proof}

Next, we show that $m(g,u) = m(g,u,\psi)$ does not depend on $u$ for fixed $\psi$. To show this, we aim to prove Theorem \ref{theorem-coordinates}. In fact, we will prove the following refinement of Theorem \ref{theorem-coordinates} which shows that $m(g,u)$ can be expressed purely in terms of $\psi$ and the metric coefficients near infinity.

\begin{theorem} \label{theorem-coordinates2} Let $g$ be a continuous asymptotically flat metric on $\R^3$ with non-negative scalar curvature in the sense of approximations. Let $u$ be a Green's function for $\lap_g$. 
Let $\an(r) = B(0,4r)-B(0,r)$. For every smooth, non-negative bump function $\psi\not\equiv 0 \in C^\infty_c((0,1))$, there is a function $\eta = \eta_\psi \in C^\infty_c((1,4))$ explicitly computable from $\psi$ such that    the harmonic mass satisfies
    \begin{align*}
        \left[3\pi \int_0^1 \frac{\psi(s)}{(1+s)^2}\, ds\right]m(g,u,\psi) &= \lim_{r\to \infty}\bigg[ \frac 1 r \int_{\an(r)} \left(\frac{1}{\vert x\vert} \eta\left(\frac{\vert x\vert}{r}\right) + \frac 1 r \eta'\left(\frac{\vert x\vert}{r}\right) \right)  \delta^{ij}(g_{ij}-\delta_{ij}) \,dx  \\
&\qquad \qquad  + \frac 1 r \int_{\an(r)} \left(\frac{1}{\vert x\vert} \eta\left(\frac{\vert x\vert}{r}\right) - \frac 1 r \eta'\left(\frac{\vert x\vert}{r}\right)\right) (g_{ij}-\delta_{ij}) \frac{x^ix^j}{\vert x\vert^2}\,dx\bigg].
\end{align*} 
In particular, $m(g,u,\psi)$ does not depend on $u$ for fixed $\psi$.  
\end{theorem}

Let $u$ be a Green's function for $\lap_g$ with pole at $x_0$ and fix an admissible $\psi$.  Recall that we are assuming the decay rate $\vert g(x) - g_{\text{euc}}(x)\vert = O(\vert x\vert^{-\tau})$ for some $\tau > \frac 1 2$. The first step is to obtain convergence of $r u(rx)$ to $\vert x\vert^{-1}$ with a suitable rate. Before reading the proof, the reader may wish to review Appendix \ref{Appendix-Mapping}. 

\begin{prop}
\label{prop-decay}
We have 
\[
\left\|r u(rx) - \frac{1}{\vert x\vert}\right\|_{W^{1,p}(\an)} = O(r^{-\beta}). 
\]
for some $\beta > 1/2$. 
\end{prop}

\begin{proof} Let $u_r(x) = ru(rx)$. Let $A(x) = g^{ij}\sqrt{\det g}$ and let $A_r(x) = A(rx)$. 
First we show that $u_{r}\to \vert x\vert^{-1}$ in $W^{1,p}(\R^3-\{0\})$. 
Note $u_r$ solves 
\[
\div(A_r(x)\grad u_r) = 0
\]
on $\R^{3} -\{x_0/r\}$. Since $A_r\to I$ uniformly as $r\to \infty$, Meyer's $L^p$ gradient estimate \cite[Theorem 2]{meyers1963p} implies that 
\[
\|\grad u_{r}\|_{L^p(\an_2)} \le C \|u_{r}\|_{L^2(\an_1)} \le C
\]
for any $\an_2 \subset \an_1$. 
Passing to a subsequence and using a diagonal argument, we have $u_{r_k}\weak w$ weakly in $W^{1,p}(\R^3-\{0\})$. Moreover, $w$ solves the equation $\div(\grad w) = 0$ in the weak sense and so $w$ is smooth and harmonic in $\R^3 - \{0\}$. The function $w$ also inherits the $C^0$ bounds 
\[
\frac{c}{\vert x\vert} \le w(x) \le \frac{C}{\vert x\vert }
\]
and it follows that $w(x) = a\vert x\vert^{-1}$ for some constant $a$. Using Meyer's $L^p$ gradient estimate again \cite[Theorem 2]{meyers1963p}, one can show that $u_{r_k}\to w$ strongly in $W^{1,p}_{\text{loc}}(\R^3 - \{0\})$. One then has convergence of the fluxes and it follows that $a = 1$. Since the limit is independent of the subsequence, we have $u_{r} \to \vert x\vert^{-1}$ in $W^{1,p}(\an)$ as $r\to \infty$. 

Next, we look to obtain the rate. Fix any $\frac 1 2 < \beta < \min\{\tau,1\}$. Define $f = u(x) - \vert x\vert^{-1}$. Let $\an(r) = B(0,4r)-B(0,r)$. First, we aim  to show that 
\[
r^{1-\frac 3 p} \left[\int_{\an(r)} \vert f\vert^p\, dx\right]^{1/p} \le C r^{-\beta}. 
\]
Equivalently, after rescaling to the fixed annulus $\an$ via $x = ry$, we want to show 
\[
\left[\int_{\an} \left\vert u_r(y) - \frac{1}{\vert y\vert }\right\vert^p  \, dy\right]^{1/p} \le   Cr^{-\beta}. 
\]
Observe that $f$ solves 
\[
\lap f = \div(\grad u) = \div((I-A)\grad u) = \div(B(x) \grad u)
\]
where $B(x) = (I-A)(x)$ satisfies $\vert B(x)\vert = O(\vert x\vert^{-\tau})$. 
Let $X$ be a vector field on $\R^3$ which is identically 0 on $B_1$ and which agrees with $B(x)\grad u$ outside a compact set. We know that $\grad u_r$ is uniformly bounded in $L^p(\an)$: 
\[
\int_{\an} \vert \grad u_r(y)\vert^p\, dy \le C. 
\]
Equivalently, we have 
\[
\int_{\an(r)} \vert \grad u(x)\vert^p \, dx \le C r^{3-2p}. 
\]
It follows that 
\[
\|X\|_{L^p(\an(r))} \le r^{-\tau} \|\grad u\|_{L^p(\an(r))} \le C r^{3/p-2-\tau}. 
\]
This implies that $X$ belongs to a weighted $L^p$ space: 
\begin{align*}
    \int_{\R^3 - B_1} \vert X\vert^p (1+\vert x\vert)^{(2+\beta)p-3}\, dx &\le C \sum_{k=0}^\infty (4^k)^{(2+\beta)p-3}\int_{\an(4^k)} \vert X\vert^p\, dx\\
    &\le C \sum_{k=0}^\infty (4^k)^{(\beta-\tau)p}
\end{align*}
which converges since $\beta < \tau$.

In the notation of Proposition \ref{weight-estimate-I}, we have $X \in L^p_{\delta-1}$ for $\delta = -1 - \beta$. Since $-2 < \delta < -1$, we can apply Proposition \ref{weight-estimate-I} to deduce the existence of a weak solution $w$ to $\lap w = \div(X)$ satisfying the estimate 
\begin{equation}
\label{bound}
\int_{\R^3} \vert w\vert^p (1+\vert x\vert)^{p(1+\beta)-3}\, dx + \int_{\R^3} \vert \grad w\vert^p (1+\vert x\vert)^{p(2+\beta)-3}\, dx \le C\|X\|_{L^p_{\delta-1}}. 
\end{equation}
Now consider the function $h = f - w$, which is harmonic outside a compact set. Letting $w_r(x) = rw(rx)$, we have 
\[
r^{p\beta} \int_{\an} \vert w_r(y)\vert^p \,dy = r^{p(1+\beta)-3}\int_{\an(r)} \vert w(x)\vert^p\, dx \le \int_{\an(r)} \vert w\vert^p (1+\vert x\vert)^{p(1+\beta)-3}\, dx. 
\]
Therefore estimate \eqref{bound} implies that  
\[
r^\beta \|w_r\|_{L^p(\an)} \to 0
\]
as $r\to \infty$.  By the already established convergence of $f$, we also have 
\[
\|h_r\|_{L^p(\an)} \to 0
\]
as $r\to \infty$ where $h_r(x) = rh(rx)$. Now $h$ has an expansion near infinity of the form $h(x) = a + b\vert x\vert^{-1} + O(\vert x\vert^{-2})$ and the above estimate for $h_r$ implies that necessarily $a = 0$ and $b = 0$. Thus $h(x) = O(\vert x\vert^{-2})$. In particular, $h$ also satisfies $\|h_r\|_{L^p(\an)} \le C r^{-\beta}$. Therefore, we obtain the $L^p$ decay 
\[
\|f\|_{L^p(\an(r))}\le C r^{3/p-1-\beta}.
\]
Again this implies that 
\[
\left[\int_{\an} \left\vert u_r(y) - \frac{1}{\vert y\vert}\right\vert^p \, dy\right]^{1/p} \le C r^{-\beta},
\]
as desired. Finally, the gradient estimate 
\[
\left[\int_{\an} \left\vert \grad u_r(y) - \grad \frac{1}{\vert y\vert}\right\vert^p\, dy\right]^{1/p}\le Cr^{-\beta}
\]
is obtained similarly using the estimate \eqref{bound} for $w$ and the explicit asymptotic expansion for $h$. 
\end{proof}

Now we can continue with the proof of Theorem \ref{theorem-coordinates2}. Let $\rho(x) = \vert x\vert^{-1}$.  Write $D(r) = \mathcal D(g_r, u_r)$ where $g_r(x) = g(rx)$ and $u_r(x) = ru(rx)$, and 
\[
\mathcal D(h,f) = c_\psi + \int \varphi(f) \vert \grad^h f\vert^3\, dv_h
\]
where $\varphi(f) = \theta(1,f)$ with $\theta$ as in \eqref{eqn-theta}.
It follows from Proposition \ref{D-Frechet} that $\mathcal D$ is Frechet differentiable at $(\delta,\rho)$ with 
\begin{align*} 
\mathcal D(g_r,u_r) &= L(g_r - \delta,u_r - \rho) + O((\|g_r-\delta\|_{C^0}^2 + \|g_r-\delta\|_{C^0} \|u_r-\rho\|_{W^{1,p}} + \|u_r-\rho\|_{W^{1,p}}^2)\\
&= L(g_r-\delta,u_r-\rho) + O(r^{-2\beta}). 
\end{align*}
Hence multiplying by $r$ we have $\lim_{r\to \infty} rD(r) = \lim_{r\to \infty} r L(g_r-\delta,u_r-\rho)$.  Then using the formula for $L$ from Proposition \ref{D-Frechet} we have 
\begin{align*}
r L(g_r-\delta,u_r-\rho) &= r \int_{\an} \varphi'(\rho) \vert \grad \rho\vert^3 (u_r-\rho)\, dx\\
&\qquad + 3r \int_{\an} \varphi(\rho) \vert \grad \rho\vert \la \grad \rho, \grad(u_r-\rho)\ra\, dx\\
&\qquad + \frac r 2 \int_{\an} \varphi(\rho) \left[(\tr(g_r -\delta) \vert \grad \rho\vert^3 - 3 \vert \grad \rho\vert (g_r-\delta)(\grad \rho, \grad \rho)\right]\, dx. 
\end{align*} 
Let us focus attention on the integral on the second line. Integrating by parts and using the compact support of $\varphi(\rho)$, we see that  
\begin{align*} 
& r \int_{\an} \varphi(\rho) \vert \grad \rho\vert \la \grad \rho, \grad(u_r-\rho)\ra\, dx \\
&\qquad = -\int_{\an} r (u_r-\rho)\div(\varphi(\rho) \vert \grad \rho\vert \grad \rho)\, dx \\
&\qquad = -\int_{\an} r(u_r-\rho) \left[\varphi'(\rho) \vert \grad \rho\vert^3 + \varphi(\rho) \la \grad \vert \grad \rho\vert, \grad \rho\ra\right]\, dx. 
\end{align*} 
Combined with the integral on the first line, we get 
\begin{align*} 
& r \int_{\an} \varphi'(\rho) \vert \grad \rho\vert^3 (u_r-\rho)\, dy + 3r \int_{\an} \varphi(\rho) \vert \grad \rho\vert \la \grad \rho, \grad(u_r-\rho)\ra\, dy\\
&\qquad =  \int_{\an} \varphi'\left(\frac 1 {\vert x\vert}\right) \frac{1}{\vert x\vert^6} w_r\, dy - 3\int_{\an} \left[\varphi'\left(\frac{1}{\vert x\vert}\right) \frac{1}{\vert x\vert^6} + \varphi\left(\frac{1}{\vert x\vert}\right) \frac{2}{\vert x\vert^5}\right] w_r\, dx\\
&\qquad = \int_{\an} \left[-6 \varphi\left(\frac 1 {\vert x\vert} \right) \frac 1{\vert x\vert^5} - 2 \varphi'\left(\frac{1}{\vert x\vert}\right) \frac{1}{\vert x\vert^6}\right] w_r\, dx.  
\end{align*} 
where $w_r = r(u_r-\rho)$.

We can choose a smooth, radial function $\chi$ compactly supported in $\an$ so that 
\[
\lap \chi = -6 \varphi\left(\frac 1 {\vert x\vert}\right) \frac 1{\vert x\vert^5} - 2\varphi'\left(\frac{1}{\vert x\vert}\right) \frac{1}{\vert x\vert^6}.
\]
Indeed, to see that this is possible, denote the right hand side above by $\Theta(s)$ where $s = \vert x\vert$. Note that the Laplacian of a radial function $\chi(s)$ is given by 
\[
\lap \chi = \frac{1}{s^2}(s^2\chi'(s))'.
\]
Therefore, every solution solution to $\lap \chi = \Theta(s)$ satisfies 
\[
\chi'(s) = s^{-2} \left[\int_0^s t^2\Theta(t)\, dt\right].
\]
\begin{align*}
\int _0^4t^2\Theta(t)\, dt =\int_0^\infty t^2 \Theta(t)\, dt &= \int_0^\infty -6 t^{-3} \varphi(t^{-1}) - 2 t^{-4} \varphi'(t^{-1})\, dt \\
&= \int_0^\infty -6 q \varphi(q) - 2 q^2 \varphi'(q)\, dq \\
&= \int_0^\infty -2q\varphi(q) -\frac{d}{dq} (2q^2 \varphi(q))\, dq \\
&= \int_0^\infty -2q\varphi(q) \, dq
\end{align*}
and we further have 
\begin{align*}
\int_0^\infty q\varphi(q) \, dq &= \int_0^\infty q^{-2} \left[\frac 1 2 \psi\left(\frac{1}{2q}-1\right) - \psi\left(\frac{1}{q}-1\right)\right]\, dq \\
&= \int_0^\infty \frac 1 2 \psi\left(\frac{t}{2}-1\right) - \psi(t-1)\, dt = 0.
\end{align*}  
Thus $\chi'$ is compactly supported in $(1,4)$. Now define
\[
\chi(s) = \int_0^s t^{-2} \left[\int_0^t v^2\Theta(v)\, dv\right]\, dt.
\]
It is easy to see $\chi(1) = 0$. We also have  
\begin{align*}
\chi(4) = \lim_{t\to\infty}\chi(t) &= \int_0^\infty t^{-2} \left[\int_{0}^t v^2\Theta(v)\, dv\right]\,dt \\
&= \int_0^\infty v^2\Theta(v) \left[\int_v^\infty  t^{-2}\, dt\right]\, dv \\
&=  \int_0^\infty v\Theta(v)\, dv. 
\end{align*}  
Then we further calculate that 
\begin{align*}
\int_0^\infty v \Theta(v)\, dv &= \int_0^\infty -6 v^{-4} \varphi(v^{-1}) - 2 v^{-5} \varphi'(v^{-1})\, dv \\
&= \int_0^\infty -6q^2 \varphi(q) - 2 q^3 \varphi'(q)\, dq \\
&= \int_0^\infty -\frac{d}{dq}(2 q^3 \varphi(q))\, dq = 0
\end{align*}
where we made the substitution $q = v^{-1}$ and used the compact support of $\varphi$.  Thus $\chi(4) = 0$ as well, and we have demonstrated the existence of $\chi$. 

We now obtain 
\begin{align*}
\int_{\an} w_r \lap \chi\, dx &= - \int_{\an} \la \grad \chi, \grad w_r\ra\, dx.
\end{align*}
Write $A_r = g_r^{ij} \sqrt{\det g_r}$ so that $\div(A_r(x) \grad u_r) = 0$ in the weak sense. We can write $A_r = I + B_r$. Then note that 
\begin{align*} 
-\int_{\an} \la \grad \chi, \grad w_r\ra \, dx &= -\int_{\an} r \la \grad \chi, (A_r + I - A_r)\grad u_r - \grad \rho \ra \,dx \\
&= \int_{\an} \la \grad \chi, r B_r \grad u_r\ra \, dx \\
&= \int_{\an} \la \grad \chi, rB_r(\grad u_r - \grad \rho) + r B_r \grad \rho\ra\, dx \\
&= \int_{\an} \la \grad \chi, rB_r \grad \rho\ra\, dx  + o(1). 
\end{align*} 
We used the fact that $\chi$ is compactly supported in $\an$ from the first line to the second line.
We have an expansion
\[
B_r = g_r^{ij} \sqrt{\det g_r} - \delta^{ij} = \frac 1 2 (\tr k_r) I - k_r + O(\|k_r\|_{C^0}^2)
\]
where $k_r = g_r - \delta$.  This yields 
\begin{align*}
\int_{\an}\la \grad \chi, r B_r \grad \rho\ra \, dx &= r \int_{\an}\left( \la \grad \chi, \frac 1 2 (\tr k_r) \grad \rho\ra -k_r(\grad \chi, \grad \rho)\right) \, dx + o(1).
\end{align*}
Then returning to the original formula and writing $s=\vert x\vert$ we have 
\begin{align*} 
r L(g_r -\delta, u_r - \rho) &=  \int_{\an} \frac r 2 (\tr k_r) \la \grad \chi, \grad \rho\ra -  r k_r(\grad \chi, \grad \rho) \, dx \\
&\qquad + \frac 1 2 \int_{\an} r \varphi(\rho) \tr(k_r) \vert \grad \rho\vert^3 - 3r \varphi(\rho) \vert \grad \rho\vert k_r(\grad \rho,\grad \rho) \, dx\\
&= \int_{\an} \left(\frac{1}{2}s^{-6} \varphi\left(\frac 1 s\right) - \frac 1 2 s^{-2} \chi'(s)\right) r \tr k_r \, dx\\
&\qquad + \int_{\an} \left(-\frac 3 2 s^{-6} \varphi\left(\frac 1 s\right) +  s^{-2} \chi'(s) \right) r k_r(\bd_s,\bd_s)\, dx.
\end{align*} 
Now define the function 
\[
\eta(s) = -\frac 1 2 s^{-5} \varphi(s^{-1}) + \frac 1 4 s^{-1}\chi'(s). 
\]
Then $\eta$ is compactly supported in $(1,4)$ and 
\begin{align*} 
\eta'(s) &= \frac 5 2 s^{-6} \varphi(s^{-1}) + \frac 1 2 s^{-7} \varphi'(s^{-1}) - \frac 1 4 s^{-2}\chi'(s) + \frac 1 4 s^{-1} \chi''(s) \\
&= \frac 5 2 s^{-6} \varphi(s^{-1}) + \frac 1 2 s^{-7} \varphi'(s^{-1}) - \frac 1 4 s^{-2}\chi'(s) + \frac 1 4 s^{-1} \left[-2 s^{-6}\varphi'(s^{-1}) - 6 s^{-5} \varphi(s^{-1}) - 2s^{-1}\chi'(s)\right]\\
&= s^{-6} \varphi(s^{-1}) - \frac 3 4 s^{-2} \chi'(s),
\end{align*} 
where we used the equation satisfied by $\chi$. In particular, it follows that 
\begin{align*}
r L(g_r-\delta,u_r-\rho) = \int_{\an} \left(s^{-1}\eta(s) + \eta'(s) \right) r\tr k_r + \left(s^{-1}\eta(s) - \eta'(s)\right) rk_r(\bd_s,\bd_s)\, dx. 
\end{align*}
Writing this back on $\an(r)$, we have 
\begin{align*}
r L(g_r-\delta,u_r-\rho) = &\int_{\an(r)} \left(\frac{r}{\vert x\vert} \eta\left(\frac{\vert x\vert}{r}\right) + \eta'\left(\frac{\vert x\vert}{r}\right) \right) r \delta^{ij}(g_{ij}-\delta_{ij})\, \frac{dx}{r^3} \\
&\qquad + \int_{\an(r)} \left(\frac{r}{\vert x\vert} \eta\left(\frac{\vert x\vert}{r}\right) - \eta'\left(\frac{\vert x\vert}{r}\right)\right) r (g_{ij}-\delta_{ij}) \frac{x^ix^j}{\vert x\vert^2}\, \frac{dx}{r^3}\\
&= \frac 1 r \int_{\an(r)} \left(\frac{1}{\vert x\vert} \eta\left(\frac{\vert x\vert}{r}\right) + \frac 1 r \eta'\left(\frac{\vert x\vert}{r}\right) \right)  \delta^{ij}(g_{ij}-\delta_{ij}) \,dx  \\
&\qquad + \frac 1 r \int_{\an(r)} \left(\frac{1}{\vert x\vert} \eta\left(\frac{\vert x\vert}{r}\right) - \frac 1 r \eta'\left(\frac{\vert x\vert}{r}\right)\right) (g_{ij}-\delta_{ij}) \frac{x^ix^j}{\vert x\vert^2}\,dx.
\end{align*} 
Thus we have 
\begin{align*}
\lim_{r\to \infty} rD(r) &= \lim_{r\to \infty} \bigg[\frac 1 r \int_{\an(r)} \left(\frac{1}{\vert x\vert} \eta\left(\frac{\vert x\vert}{r}\right) + \frac 1 r \eta'\left(\frac{\vert x\vert}{r}\right) \right)  \delta^{ij}(g_{ij}-\delta_{ij}) \,dx  \\
&\qquad\qquad + \frac 1 r \int_{\an(r)} \left(\frac{1}{\vert x\vert} \eta\left(\frac{\vert x\vert}{r}\right) - \frac 1 r \eta'\left(\frac{\vert x\vert}{r}\right)\right) (g_{ij}-\delta_{ij}) \frac{x^ix^j}{\vert x\vert^2}\,dx\bigg].
\end{align*}
This completes the proof of Theorem \ref{theorem-coordinates2}.

\section{Rigidity}
\label{section-rigidity}

In this final section, we prove the rigidity part of the positive mass theorem. Let $g$ be a continuous asymptotically flat metric on $\R^3$ with non-negative scalar curvature in the sense of approximations. 
Assume that $m(g) = 0$. 
Note this implies that $m(g,u,\psi) = 0$ for all choices of $u$ and $\psi$.  Hence $D(r) = 0$ for all choices of $u$ and $\psi$ and $r$. Fix some $u$ and $\psi$.

\begin{prop}
    There is a constant $c$ such that $\vert \grad^g u\vert = c u^2$ almost-everywhere. 
\end{prop}

\begin{proof} Fix some $d > 0$. Without loss of generality we can suppose the pole of $u$ is at the origin. Choose a Euclidean annulus $\an$ which compactly contains $\{1/(8d) < u < 1/d\}$. Choose smooth metrics $g_k$ with $R(g_k)\ge -\eps_k$ converging uniformly to $g$ on $\an$. Let $u_k$ be the $g_k$-harmonic function which agrees with $u$ on the boundary of $\an$. Then we have $D_k(r) \to D(r) = 0$ as $k\to \infty$ for each fixed $r\in [d,2d]$. Writing $D_k(r)$ as an integral of $F_k$ as in Proposition \ref{DvsF}, it follows that $F_k \to 0$ as an $L^1$ function on $[d,8d]$. Hence we can find $a \approx d$ and $b \approx 8d$ such that $F_k(a) \to 0$ and $F_k(b)\to 0$ as $k\to \infty$. 
    
By the derivative formula Proposition \ref{F-derivative}, the fact that the level sets of $u_k$ are connected, and the fact that $\eps_k\to 0$, it follows that
\begin{equation}
\label{stability-estimate}
\int_a^b \left[\int_{\{u_k = 1/t\}} \frac{\vert \grad^{g_k,T} \vert \grad^{g_k} u_k\vert\vert^2}{\vert \grad^{g_k} u_k\vert^2} + \left(H_k - \frac{2 \vert \grad^{g_k} u_k\vert}{u_k}\right)^2\, da_k\right]\, dt \to 0.
\end{equation}
Choose a bounded, connected open set $\Omega$ with smooth boundary which is a subset of $\{1/b < u_k < 1/a\}$ for all $k$ and which contains $\{1/(7d) < u < 1/(2d)\}$. Since the metrics $g_k$ are all uniformly comparable (as they converge uniformly to $g$), there are uniform Poincaré inequalities on the fixed region $\Omega$:
\[
\int_\Omega \vert f - \overline f\vert \, dv_k \le C \int_\Omega \vert \grad^{g_k} f\vert\, dv_k \le C \int_{\{1/b < u_k < 1/a\}} \vert \grad^{g_k} f\vert\, dv_k
\]
where $\overline f$ is the average of $f$ on $\Omega$ with respect to $g_k$ and $C$ does not depend on $k$. 

Let $U_k = \{1/b < u_k < 1/a\}$. This Poincaré inequality implies that there are constants $c_k$ so that  
\begin{align*} 
\int_\Omega \left\vert\frac{\vert \grad^{g_k} u_k\vert}{u_k^2} - c_k\right\vert\, dv_k &\le C \int_{U_k} \left\vert \grad^{g_k} \left(\frac{\vert \grad^{g_k} u_k\vert}{u_k^2}\right)\right\vert\, dv_k\\
&= C \int_{U_k} \left\vert \frac{\grad^{g_k} \vert \grad^{g_k} u_k\vert}{u_k^2} - 2 \frac{\vert \grad^{g_k} u_k\vert}{u_k^3} \grad^{g_k} u_k\right\vert\, dv_k\\
&\le C \int_{U_k} \left\vert {\grad^{g_k} \vert \grad^{g_k} u_k\vert} - 2 \frac{\vert \grad^{g_k} u_k\vert}{u_k} \grad^{g_k} u_k\right\vert\, dv_k
\end{align*}
where in the last line we used the fact that $u_k\in [a,b]$. Next, observe that 
\begin{align*}
&\int_{U_k} \left\vert {\grad^{g_k} \vert \grad^{g_k} u_k\vert} - 2 \frac{\vert \grad^{g_k} u_k\vert}{u_k} \grad^{g_k} u_k\right\vert\, dv_k \\
&\qquad \le  C \int_a^b \int_{\{u_k=1/t\}} \left\vert {\grad^{g_k} \vert \grad^{g_k} u_k\vert} - 2 \frac{\vert \grad^{g_k} u_k\vert}{u_k} \grad^{g_k} u_k\right\vert\, \frac{1}{\vert \grad^{g_k} u_k\vert}\,  da_k \, dt\\
&\qquad = C \int_a^b \int_{\{u_k = 1/t\}}  \left\vert \frac{\grad^{g_k} \vert \grad^{g_k} u_k\vert}{\vert \grad^{g_k} u_k\vert} - 2 \frac{\vert \grad^{g_k} u_k\vert}{u_k} \frac{\grad^{g_k} u_k}{\vert \grad^{g_k} u_k\vert} \right\vert\, da_k\, dt\\
&\qquad \le C \int_a^b \left[\int_{\{u_k=1/t\}} \left\vert \frac{\grad^{g_k} \vert \grad^{g_k} u_k\vert}{\vert \grad^{g_k} u_k\vert} - 2 \frac{\vert \grad^{g_k} u_k\vert}{u_k} \frac{\grad^{g_k} u_k}{\vert \grad^{g_k} u_k\vert} \right\vert^{2}\, da_k\right]^{1/2} \left[\int_{\{u_k=1/t\}} 1\, da_k\right]^{1/2}\, dt \\
&\qquad \le C \left[\int_a^b \int_{\{u_k=1/t\}} \left\vert \frac{\grad^{g_k} \vert \grad^{g_k} u_k\vert}{\vert \grad^{g_k} u_k\vert} - 2 \frac{\vert \grad^{g_k} u_k\vert}{u_k} \frac{\grad^{g_k} u_k}{\vert \grad^{g_k} u_k\vert} \right\vert^{2}\, da_k\, dt\right]^{1/2} \left[\int_a^b \int_{\{u_k=1/t\}} 1 \, da_k \, dt\right]^{1/2}
\end{align*} 
where again we used the fact that $u_k \in [a,b]$.  By the co-area formula we have 
\[
\int_a^b \int_{\{u_k=1/t\}} 1 \, da_k \, dt \le C \int_{U_k} \vert \grad^{g_k} u_k\vert\, dv_k 
\]
and this is uniformly bounded. Now observe that 
\[
\int_a^b \int_{\{u_k=1/t\}} \left\vert \frac{\grad^{g_k} \vert \grad^{g_k} u_k\vert}{\vert \grad^{g_k} u_k\vert} - 2 \frac{\vert \grad^{g_k} u_k\vert}{u_k} \frac{\grad^{g_k} u_k}{\vert \grad^{g_k} u_k\vert} \right\vert^{2}\, da_k\, dt \to 0. 
\]
Indeed, this follows from the estimate 
\begin{align*}
&\frac 1 2 \left\vert \frac{\grad^{g_k} \vert \grad^{g_k} u_k\vert}{\vert \grad^{g_k} u_k\vert} - 2 \frac{\vert \grad^{g_k} u_k\vert}{u_k} \frac{\grad^{g_k} u_k}{\vert \grad^{g_k} u_k\vert} \right\vert^{2} \\
&\qquad \le \left\vert \frac{\grad^{g_k,T} \vert \grad^{g_k} u_k\vert}{\vert \grad^{g_k} u_k\vert}\right\vert^2 + \left\vert \frac{\grad^{g_k,N} \vert \grad^{g_k} u_k\vert}{\vert \grad^{g_k} u_k\vert} - 2 \frac{\vert \grad^{g_k} u_k\vert}{u_k} \frac{\grad^{g_k} u_k}{\vert \grad^{g_k} u_k\vert}\right\vert^2 \\
&\qquad \le \left\vert \frac{\grad^{g_k,T} \vert \grad^{g_k} u_k\vert}{\vert \grad^{g_k} u_k\vert}\right\vert^2 +\left\vert \left\la \frac{\grad^{g_k} \vert \grad^{g_k} u_k\vert}{\vert \grad^{g_k} u_k\vert}, \frac{\grad^{g_k} u_k}{\vert \grad^{g_k} u_k\vert}\right\ra \frac{\grad^{g_k} u_k}{\vert \grad^{g_k} u_k\vert} - 2 \frac{\vert \grad^{g_k} u_k\vert}{u_k} \frac{\grad^{g_k} u_k}{\vert \grad^{g_k} u_k\vert}\right\vert^2\\
&\qquad=  \left\vert \frac{\grad^{g_k,T} \vert \grad^{g_k} u_k\vert}{\vert \grad^{g_k} u_k\vert}\right\vert^2 + \left\vert H_k -  \frac{2\vert \grad^{g_k} u_k\vert}{u_k}\right\vert^2 
\end{align*} 
together with \eqref{stability-estimate}. Hence sending $k$ to infinity in the Poincaré inequality estimate, we deduce that there is a constant $c$ such that $\vert \grad^g u\vert = cu^2$ almost-everywhere on $\Omega$. Finally, the result follows since $d$ was arbitrary.  
\end{proof}

\begin{prop}
We have $\vert \grad^g u\vert = u^2$ almost-everywhere. 
\end{prop}

\begin{proof}
    From the previous proposition, we know that $\vert \grad^g u\vert = cu^2$ almost-everywhere for some constant $c$. We also know that $D(r) = 0$ for every  $r > 0$. Observe that 
    \begin{align*}
        D(r) &= 2\pi \int_0^1 \frac{\psi(s)}{1+s}\, ds + \int \left[\frac 1 2 \psi\left(\frac{1}{2 r u}-1\right)-\psi\left(\frac{1}{ru}-1\right)\right] \frac{\vert \grad^g u\vert^3}{u^3}\, dv_g\\
        &= 2\pi \int_0^1 \frac{\psi(s)}{1+s}\, ds + \int_{\{1/(4r)<u<1/r\}} \left[\frac 1 2 \psi\left(\frac{1}{2 r u}-1\right)-\psi\left(\frac{1}{ru}-1\right)\right] \frac{c}{u}  \vert \grad^g u\vert^2 \, dv_g. 
    \end{align*}
    By the Sobolev co-area formula, this implies 
    \begin{align*}
        D(r) &= 2\pi \int_0^1 \frac{\psi(s)}{1+s}\, ds + \int_{1/(4r)}^{1/r} \left[\frac 1 2 \psi\left(\frac{1}{2r t}-1\right) - \psi\left(\frac 1 {rt}-1\right)\right] \frac{c}{t}\left[\int_{\{u=t\}} \vert \grad^g u\vert\, d\mathcal H^2_g\right]\, dt \\
        &= 2\pi \int_0^1 \frac{\psi(s)}{1+s}\, ds + \int_{1/(4r)}^{1/r} \left[\frac 1 2 \psi\left(\frac{1}{2r t}-1\right) - \psi\left(\frac 1 {rt}-1\right)\right] \frac{4\pi c}{t}\, dt.
    \end{align*}
    By suitable changes of variables, this implies that   
    \begin{align*} 
 0 = D(r) = 2\pi (1-c)\int_0^1  \frac{\psi(s)}{1+s}\, ds.
    \end{align*}
    It follows that $c= 1$. 
\end{proof}

Next, define $\rho = u^{-1}$. Then it is easy to compute that 
\[
\vert \grad^g \rho\vert = \frac{\vert \grad^g u\vert}{u^2} = 1
\]
almost everywhere. In fact, $\rho^2$ solves a very simple equation on all of $\R^3$. 

\begin{prop} 
The function $\rho^2$ satisfies $\lap_g(\rho^2) = 6$ in the sense of distributions on all of $\R^3$ (even across the pole). 
\end{prop} 

\begin{proof}
Without loss of generality, suppose the pole $x_0$ is at the origin.
Let $A(x) = g^{ij}\sqrt{\det g}$. Since $u$ is $g$-harmonic away from the pole, we know that 
\[
\int\la A(x)\grad u,\grad \varphi\ra \, dx = 0
\]
for all smooth functions $\varphi$ compactly supported in $\R^3 -\{0\}$. In fact, we have $u^{-3}\varphi\in W^{1,2}$ and so we have 
\begin{align*}
\int \la A(x) \grad u, \grad(-2 u^{-3} \varphi)\ra\, dx = 0
\end{align*}
for all such $\varphi$ by approximation. On the other hand, note that $\grad (\rho^2) = \grad(u^{-2}) = -2 u^{-3}\grad u$ and so 
\begin{align*}
    \int \la A(x)\grad u, \grad(-2 u^{-3}\varphi)\ra\, dx&= \int \la A(x)\grad u, 6u^{-4}\varphi\grad u\ra \,dx + \int\la A(x) \grad u,-2 u^{-3}\grad \varphi\ra\, dx\\
&= \int 6u^{-4}  \vert \grad^g u\vert^2 \varphi\, dv_g + \int \la A(x) \grad (\rho^2), \grad \varphi\ra\, dx\\&= \int 6 \varphi\, dv_g + \int g(\grad^g(\rho^2),\grad^g \varphi)\, dv_g.  
\end{align*}
Hence we have 
\[
\int g(\grad^g(\rho^2),\grad^g\varphi)\, dv_g = -\int 6\varphi\, dv_g
\]
for all smooth functions $\varphi$ with compact support in $\R^3 - \{0\}$ and so $\rho^2$ solves $\lap_g(\rho^2) = 6$ in the sense of distributions away from the pole. 

It remains to extend this across the pole. Assume now that $\varphi$ is a smooth compactly supported function on $\R^3$. Let $\eta$ be a a smooth function which is 1 outside $B_{2r}$ and 0 inside $B_r$ and satisfies $\vert \grad \eta\vert \le Cr^{-1}$. Then using $\eta \varphi$ as a test function we see that 
\[
\int g(\grad^g(\rho^2),\eta \grad^g \varphi)\, dv_g + \int g(\grad^g(\rho^2),\varphi\grad^g \eta)\, dv_g = -\int 6\eta\varphi\, dv_g. 
\]
Now note that $\vert \grad^g (\rho^2)\vert = 2\rho \vert \grad^g \rho\vert = 2\rho$ almost-everywhere. Also since $u$ goes to infinity at the pole, and $\rho$ satisfies $\vert \grad^g \rho\vert = 1$ almost-everywhere, it follows that there is a local Lipschitz bound $\vert \rho(x)\vert = \vert \rho(x)-\rho(0)\vert \le C \vert x\vert$. Hence we have $\vert \grad^g (\rho^2)\vert \le C \vert x\vert$. It now follows easily that 
\begin{gather*}
    \int g(\grad^g(\rho^2),\eta \grad^g \varphi)\, dv_g \to \int g(\grad^g(\rho^2),\grad^g\eta)\, dv_g,\\
    \int 6\eta\varphi\, dv_g \to \int 6\varphi\, dv_g,
\end{gather*}
as $r\to 0$. 
Finally, note that 
\[
\int \vert g(\grad^g(\rho^2),\varphi \grad^g \eta)\vert\, dv_g \le \frac{C}{r} \int_{B_{2r}-B_r} \vert x\vert \, dx \le Cr^3\to 0
\]
as $r\to 0$. Hence we obtain 
\[
\int g(\grad^g(\rho^2),\grad^g \varphi)\, dv_g = -\int 6\varphi \, dv_g
\]
for all smooth functions $\varphi$ with compact support on $\R^3$. This proves the proposition. 
\end{proof}

We now vary the choice of pole. Write $u_y$ for the Green's function with pole at $y$ and let $\rho_y = 1/u_y$.  Then for any choice of $y$, the function $\rho_y^2 - \rho_0^2$ is $g$-harmonic on all of $\R^3$.  
Moreover, the difference $\rho_y^2 - \rho_0^2$ has sub-quadratic growth at infinity. 

\begin{prop}
    The function $\rho_y^2 - \rho_0^2$ satisfies 
    \[
    \vert (\rho_y^2 - \rho_0^2)(x)\vert = O(\vert x\vert^{2-\beta}) 
    \]
    for some $\beta > \frac 1 2$. 
\end{prop}

\begin{proof}
    We have $\rho_y^2 - \rho_0^2 = (\rho_y - \rho_0)(\rho_y+\rho_0)$. The gradient estimates $\vert \grad^g \rho_y\vert \ =1$ almost-everywhere and $\vert \grad^g \rho_0\vert = 1$ almost-everywhere imply that $\vert \rho_y + \rho_0\vert$ has at most linear growth. We also have 
    \begin{align*}
        \vert \rho_y-\rho_0\vert = \frac{\vert u_0 - u_y\vert}{u_yu_0} \le C \vert x\vert^2 \vert u_0 - u_y\vert.
    \end{align*}
    Now Proposition \ref{prop-decay} implies that 
    \[
    \left\|r u_0(rx) - \frac 1 {\vert x\vert} \right \|_{W^{1,p}(\an)} = O(r^{-\beta})
    \]
    for some $\frac 1 2 < \beta < \tau$, and it follows that 
    \[
    \left\vert u_0(x) - \frac 1 {\vert x\vert}\right\vert = O(\vert x\vert^{-1-\beta}).
    \]
    Likewise we have 
    \[
    \left\vert u_y(x) - \frac 1 {\vert x\vert}\right\vert = O(\vert x\vert^{-1-\beta}).
    \]
    Combining these estimates, we obtain 
    $
    \vert u_0(x) - u_y(x)\vert = O(\vert x\vert^{-1-\beta})
    $
    and it now follows that 
    \[
    \vert (\rho_y^2 - \rho_0^2)(x)\vert \le C \vert x\vert^{2-\beta},
    \]
    as claimed. 
\end{proof}

Next, we classify the $g$-harmonic functions with this growth rate near infinity. At this point, the reader may wish to review Appendix \ref{Appendix-Mapping}. We proceed in several steps. 

\begin{prop}
\label{harmonic-coordiante}
For each coordinate function $x_i$ there is a $g$-harmonic function $w_i$ on $\R^3$ which satisfies
\[
\frac{w_i(x) - x_i}{\vert x\vert} \to 0
\]
as $\vert x\vert\to \infty$. 
\end{prop}

\begin{proof}
Let $A(x) = g^{ij}\sqrt{\det g}$ and write $A(x) = I - B(x)$ so that $\vert B(x)\vert = O(\vert x\vert^{-\tau})$.  Fix some $i = 1,2,3$. Define a vector field $F(x) = B(x) e_i$. The proposition is equivalent to the existence of a  solution $v = v_i$ to the equation
\[
\int_{\mathbb R^3}A\nabla v\cdot \nabla \varphi \, dx=\int_{\mathbb R^3}F\cdot \nabla \varphi\, dx,\quad \forall \varphi\in C^{\infty}_c(\mathbb R^3),
\]
which satisfies the decay condition  
\[
\frac{v(x)}{|x|}\to 0,\quad \text{ as } |x|\to\infty.
\]
Indeed, given such  $v = v_i$, the function $w_i = x_i + v_i$ is as required. 

Thus we aim to solve the above equation. 
The first step is to prove an a priori estimate. Fix some $p > 3$.  Let $\sigma(x) = 1+\vert x\vert$. Fix some $1 -\tau < \delta<1$.  Assume that $X$ is a vector field on $\R^3$ satisfying 
\[
\|X\|_{L^p_{\delta-1}}^p = \int_{\R^3} \vert X\vert^p \sigma^{-p(\delta-1)-3}\, dx < \infty. 
\]
According to Proposition \ref{weight-estimate-II}, there is a unique weak solution $u$ to $\lap u = \div(X)$ satisfying the normalization 
\[
\int_{B_2} u\, dx = 0. 
\]
Moreover, there is an estimate 
\[
\|u\|_{W^{1,p}_\delta} \le C \|X\|_{L^p_{\delta-1}}. 
\]
We define the bounded linear operator $T\colon L^p_\delta(\R^3)^3 \to W^{1,p}_\delta(\R^3)$ by $T(X) = u$.

We now find an approximate solution to the PDE. Choose some continuous $\widetilde A = I - \widetilde B$ so that 
\[
\|\widetilde B\|_{L^\infty} < \eps
\]
and so that $\widetilde A$ agrees with $A$ outside a compact set.  Since $\vert F(x)\vert = O(\vert x\vert^{-\tau})$, we have $F \in L^p_{\delta-1}(\R^3)^3$. Consider an operator $S \colon W^{1,p}_\delta(\R^3)\to W^{1,p}_\delta(\R^3)$ given by 
\[
S(u) = T(F+\widetilde B\grad u). 
\]
Observe that if $\tilde v$ is a fixed point of $S$ then $\lap \tilde v = \div(F + \widetilde B \grad \tilde v)$ in the weak sense. Thus $\tilde v$ solves 
\[
\div(\widetilde A \grad \tilde v) = \div(F)
\]
in the weak sense with appropriate decay. 
To find a fixed point, we aim to show that $S$ is a contraction mapping. Consider any two functions $u_1,u_2 \in W^{1,p}_\delta$. Then we have 
\begin{align*}
    S(u_1)-S(u_2) = T(\widetilde B\grad(u_1-u_2)).
\end{align*}
Now $\|\widetilde B\grad (u_1-u_2)\|_{L^p_{\delta-1}} \le \eps\|\grad(u_1-u_2)\|_
{L^p_{\delta-1}} \le \eps \|u_1-u_2\|_{W^{1,p}_\delta}$. This implies that 
\[
\| S(u_1) - S(u_2)\|_{W^{1,p}_\delta} \le C \eps \|u_1-u_2\|_{W^{1,p}_\delta}
\]
and hence $S$ is a contraction provided $\eps$ is sufficiently small. This proves the existence of a weak solution $\tilde v\in W^{1,p}_\delta$ to 
\[
\lap \tilde v = \div(F + \tilde B \grad v).
\]
Further note that since $0 < \delta < 1$ and $p>3$, the Sobolev embedding theorem \cite[Theorem 1.2 (1.10)]{bartnik1986mass} implies that any function $\tilde v \in W^{1,p}_\delta$ has sublinear growth $\tilde v = o(\vert x\vert^\delta)$. 

Finally, we need to correct $\tilde v$ to $v$. We seek $v$ in the form $v = \tilde v + u$ where $u$ solves 
\[
\div(A\grad u) = \div((\widetilde A - A)\grad \tilde v). 
\]
The vector field $Y = (\widetilde A - A)\grad \tilde v\in L^2$ is compactly supported. The bilinear form associated to this equation is coercive on the homogeneous Sobolev space $\dot H^1(\R^3)$ and so there is a weak solution $u \in \dot H^1$. By the Sobolev embedding theorem, we also have $u\in L^6$.  Then since $u$ solves $\div(A\grad u) = 0$ outside a compact set, it follows from the De-Giorgi-Nash-Moser theory \cite[Theorem 8.17]{gilbarg1998elliptic}  that $u$ is continuous and goes to 0 at infinity. Finally, the function $v = \tilde v + u$ is as required. 
\end{proof}

We want to show that the space of $g$-harmonic functions on $\R^3$ of sub-quadratic growth is spanned by $w_1$, $w_2$, and $w_3$ together with the constant functions. As a preliminary step, we prove the following Liouville type theorem for $g$-harmonic functions of sublinear growth. 

\begin{prop}\label{prop: liouville}
    Assume that $v$ is a $g$-harmonic function on $\R^3$ with sublinear growth $v = o(\vert x\vert)$. Then $v$ is constant. 
\end{prop}

\begin{proof}
    Let $v$ be a $g$-harmonic function on $\R^3$ with $v = o(\vert x\vert)$. Let $\sigma(x) = 1 + \vert x\vert$. 
We choose $\mu > 1$
which is very close to 1. 
By the sublinear assumption $v=o(\vert x\vert)$, we have that 
\[
\int_{B_{2R}-B_R}\vert v\vert^2\sigma^{-2\mu-3}\,dx=o(R^{2-2\mu}).
\]
Taking $R=2^k$ and summing over dyadic annuli, we obtain
\[
\int_{\R^3}\vert v\vert^2 \sigma^{-2\mu-3}\, dx\leq C.
\]
Let $\eta_R$ be a cut-off function on $B_{4R}-B_{R/2}$ such that $\eta_R\equiv 1$ on $B_{2R}-B_{R}$, $\eta_{R}\equiv 0$ outside $B_{4R}-B_{R/2}$ and $|\nabla \eta_R|\leq CR^{-1}$, where $C$ is a uniform constant which is independent of $R$. Plugging in the test function
\[
\varphi=\eta_R^2 v
\]
to the equation satisfied by $v$ and  using the the ellipticity of $A$ and elementary inequalities, we obtain the Caccioppoli type estimate  
\[
\int _{B_{2R}-B_R}|\nabla v|^2\, dx\leq CR^{-2}\int_{B_{4R}-B_{R/2}}v^2\, dx\leq R\, o(R^2)=o(R^{3}).
\]
It follows that  
\[
\int_{B_{2R}-B_R}|\nabla v|^2\sigma^{-2\mu-1}\, dx\leq o(R^{2-2\mu}).
\]
Taking $R=2^k$ and summing over the dyadic annuli, we obtain
\[
\|v\|_{H^1_{\mu}(\R^3)}\leq C.
\]
Next, we next need to improve the estimates. 

Choose $0 < \beta < 1$ very close to $\frac 1 2$ so that $\tau + \beta > \mu$,  and rewrite the $g$-harmonic equation as 
\[
\Delta v=\div\left(B\nabla v\right)=\div(X).
\]
Since $\vert B\vert = O(\vert x\vert^{-\tau})$, it follows that  
\[
\int_{\R^3}|X|^2\sigma^{-2\beta-1}\, dx\leq C\int_{\R^3}|\nabla v|^2\sigma^{-2(\tau+\beta)-1}\, dx\leq C\int_{\R^3}|\nabla v|^2\sigma^{-2\mu-1}\, dx\leq C,
\]
and therefore $X\in L_{\beta-1}^2(\R^3)$. 
Since $0 < \beta < 1$, Proposition \ref{weight-estimate-II} implies that there is a solution $w$ to $\lap w = \div(X)$ satisfying the estimate 
\[
\|w\|_{H^1_\beta} \le C \|X\|_{L^2_{\beta-1}}. 
\]
The difference $h = v - w$ is then a Euclidean harmonic function on $\R^3$.  Since $v$ is sublinear and $w\in L^2_\beta$ with $0<\beta < 1$, the asymptotic expansion for $h$ at infinity must be of the form 
\[
h(x) = c + O(\vert x\vert^{-1})
\]
for some constant $c$. Thus $h$ is bounded and by the classical Liouville theorem $h$ is constant. 
We now have $v = w + c$ and thus $v\in L^2_\beta$ since the constants belong to $L^2_\beta$ for $0 < \beta < 1$. This means that 
\[
\int_{\R^3} v^2 \sigma^{-2\beta - 3}\, dx \le C. 
\]
Repeating the Caccioppoli estimate, it follows that $\grad v \in L^2_{\beta - 1}$ and hence that $v \in H^1_\beta$. We have thus managed to improve the decay of the function $v$. 

We now repeat the process one more time. Assuming $\mu$ is close enough to 1 and $\beta$ is close enough to $\frac 1 2$, we can suppose that $\beta < \tau$. 
Choose a very small $\delta<0$ for which $\tau+\delta> \beta$. Then we have
\[
\int_{\R^3}|X|^2\sigma^{-2\delta-1}\, dx\leq C\int_{\mathbb R^3}|\nabla v|^2\sigma^{-2\delta-2\tau-1}\, dx\leq C,
\]
and therefore $X\in L^2_{\delta-1}(\R^3)$.  By Proposition \ref{weight-estimate-I}, there exists a solution $w_2$ to $\Delta w_2 =\div(X)$ satisfying 
\[
\|w_2\|_{H_{\delta}^1}\leq C\|X\|_{L^2_{\delta-1}}.
\]
Again, by the classical Liouville theorem, the difference $h_2 = v - w_2$ must be constant.  Thus $v = c + w_2$ for some constant $c$. It follows that $w_2$ is also $g$-harmonic with $\div(A \grad w_2) = 0$. Since $w_2 \in L^2_\delta$ with $\delta<0$, the De Giorgi-Nash-Moser theory \cite[Theorem 8.17]{gilbarg1998elliptic} implies that $w_2$ is bounded. More  precisely, for any $R>0$, we have that 
\begin{align*} 
\sup_{B_{2R}-B_R} |w_2|&\leq C R^{-3/2}\|w_2\|_{L^2(B_{4R}-B_{R/2})}\\ 
&\leq C\left(\int_{B_{4R}-B_{R/2}}w_2^{2}\sigma^{-3}\, dx\right)^{1/2}\\
&\leq CR^{\delta}\|w_2\|_{L_{\delta}^2(B_{4R})}\leq CR^{\delta}, \phantom{\int}
\end{align*}
where $C$ is a positive constant independent of $R$ since the $g$-harmonic equation is of divergence form. Since $\delta<0$, taking $R\to\infty$, we obtain the uniform boundedness of $w_2$.

Thus $v$ is a bounded entire solution to $\div(A\grad v) = 0$ and it follows again from the De Giorgi-Nash-Moser theory that $v$ must be constant (c.f. \cite[Theorem 8.22]{gilbarg1998elliptic}). This completes the proof. 
\end{proof}

We can now characterize the space of $g$-harmonic functions with at most linear growth. 

\begin{prop}
The space of $g$-harmonic functions on $\R^3$ with at most linear growth is spanned by $1$, $w_1$, $w_2$, $w_3$. 
\end{prop} 

\begin{proof} 

Suppose $v$ is a $g$-harmonic function with at most linear growth. For any $1<\mu<\min\{2,1+\tau\}$, we have that 
\[
v\in L^2_{\mu}(\R^3).
\]
By the Caccioppoli type estimates in the proof of the previous proposition, we have that 
\[
\|v\|_{H_{\mu}^1(\R^3)}\leq C.
\]
It is now straightforward to verify that 
\[
X\in L^2_{\beta-1}(\R^3),\quad \forall \max\{0,\mu-\tau\}<\beta<1.
\]
By Proposition \ref{weight-estimate-II}, we know there exists a unique weak solution to $\Delta w=\div(X)$ satisfying the normalization 
\[
\int_{B_2}w\, dx=0.
\]
Moreover, $w\in H^1_\beta$. 
By the De Giorgi-Nash-Moser theory, this implies that that $w\sim o(\vert x\vert)$ at infinity.  

Now, we let $h=v-w$. Then $h$ is a Euclidean harmonic function and $h\in H_{\mu}^1$. By the classical Liouville's theorem, we know 
\[
h=a_1x_1 + a_2 x_2 + a_3x_3 +c,
\]
for some constants $a_i, c\in\R$.
Therefore, we can  rewrite the asymptotics of $v$ as
\[
v\sim a_1x_1 + a_2x_2 + a_3 x_3 +o(\vert x\vert).
\]
By Proposition \ref{harmonic-coordiante} we see that 
\[
v-a_1w_1 -a_2w_2 - a_3w_3 \sim o(\vert x\vert)
\]
is a $g$-harmonic function with sublinear growth on $\R^3$. Therefore $v - a_1w_1 - a_2w_2 - a_3w_3$ must be a constant by Proposition \ref{prop: liouville}, and the result follows.
\end{proof} 

Finally, we extend this further to functions with sub-quadratic growth $O(\vert x\vert^{2-\beta})$. 

\begin{prop}
    Assume that $v$ is a $g$-harmonic function on $\R^3$ with $\vert v(x)\vert = O(\vert x\vert^{2-\beta})$ for some $\frac 1 2 < \beta < \tau$. Then $v \in \operatorname{span}\{1,w_1,w_2,w_3\}$. 
\end{prop}

\begin{proof}
This is a straightforward extension of the previous result. Indeed, taking $\mu = \frac 3 2$ we have $2-\beta < \mu$ and thus $v \in H^1_\mu(\R^3)$. Since $\mu < 1 +\tau$, the proof of the previous proposition now goes through unchanged. 
\end{proof}

By the previous propositions, for each $y$ we have 
\[
\rho_0^2 - \rho_y^2 = b_0(y) + b_1(y)w_1 + b_2(y)w_2 + b_3(y) w_3
\]
for some constants $b_0,b_1,b_2,b_3$.  This implies that 
\[
\grad^g (\rho_y^2) = \grad^g(\rho_0^2) - \sum_{i=1}^3 b_i(y) \grad^g w_i. 
\]
We further note that $\vert \grad^g (\rho_y^2)\vert^2 = 4\rho_y^2$ almost everywhere. 

\begin{prop}
The coefficients $b_0(y),b_1(y),b_2(y),b_3(y)$ depend continuously on $y$. 
\end{prop}

\begin{proof} 
Fix a point $y_0$ and let $U$ be a small neighborhood of $y_0$. Fix any point $q_0 \notin U\cup\{0\}$ and then let $q_1 = (r,0,0)$, $q_2=(0,r,0)$, $q_3=(0,0,r)$ for some fixed large $r > 0$. Since $w_i$ is asymptotic to $x^i$, the matrix 
\[
A = \begin{pmatrix}
    1 & w_1(q_0) & w_2(q_0) & w_3(q_0)\\ 
    1 & w_1(q_1) & w_2(q_1) & w_3(q_1)\\
    1 & w_1(q_2) & w_2(q_2) & w_3(q_2)\\
    1 & w_1(q_3) & w_2(q_3) & w_3(q_3)
\end{pmatrix}
\]
will be invertible provided $r$ is chosen large enough. Now observe that for $y\in U$ we have  
\[
A \begin{pmatrix}
    b_0(y) \\ b_1(y)\\b_2(y)\\ b_3(y)
\end{pmatrix} = \begin{pmatrix}
    (\rho_0^2 - \rho_y^2)(q_0)\\
    (\rho_0^2 - \rho_y^2)(q_1)\\
    (\rho_0^2 - \rho_y^2)(q_2)\\
    (\rho_0^2 - \rho_y^2)(q_3)
\end{pmatrix}.
\]
Since $A$ is invertible, the continuity of $b_i$ at $y_0$ will follow provided we can show the functions $y\mapsto  \rho_y^2(q_i)$ are continuous functions of $y\in U$. Since no $q_i$ belongs to $U$, it is equivalent to show the continuity of $y\mapsto u_y(q_i)$. Now the construction in \cite[Theorem (1.1)]{gruter1982green} actually supplies a 2-variable Green's function $G(x,y)$ with the property that $u_y(x) = G(x,y)$. Moreover, by \cite[Theorem (1.3)]{gruter1982green}, the 2-variable Green's functions has the symmetry $G(x,y) = G(y,x)$ and hence $u_y(q_i) = u_{q_i}(y)$.  But $y\mapsto u_{q_i}(y)$ is clearly continuous and the result follows. 
\end{proof}

\begin{prop}
The map $y\mapsto \mathbf b(y) = (b_1(y),b_2(y),b_3(y))$ is injective and its image is an open set. 
\end{prop}

\begin{proof}
Suppose for contradiction that the map is not injective. Then there are two distinct points $y_1$ and $y_2$ with $\mathbf b(y_1) = \mathbf b(y_2) = (b_1,b_2,b_3)$. This means that 
\begin{gather*}
\rho_0^2 - \rho_{y_1}^2 = b_0(y_1) + w,\\
\rho_0^2 - \rho_{y_2}^2 = b_0(y_2) + w
\end{gather*}
for the same function $w = b_1 w_1 + b_2 w_2 + b_3 w_3$. Subtracting these equations, we get that $\rho_{y_1}^2 - \rho_{y_2}^2$ is constant. But this is impossible since we have $(\rho_{y_1}^2 - \rho_{y_2}^2)(y_1) = -\rho_{y_2}^2(y_1) < 0$ and $(\rho_{y_1}^2-\rho_{y_2}^2)(y_2) = \rho_{y_1}^2(y_2) > 0$.  This proves that $y\mapsto \mathbf b(y)$ is injective. Finally, since $y\mapsto \mathbf b(y)$ is injective and continuous, invariance of domain implies that the image of this map is an open set. 
\end{proof}

Next, for each fixed $y$, we calculate that 
\begin{align*}
\left\vert \grad^g (\rho_0^2)  -  \sum_{i=1}^3 b_i(y) \grad^g w \right\vert^2 &= \vert \grad^g (\rho_y^2)\vert^2 = 4\rho_y^2\\
& = 4\rho_0^2 - 4 \sum_{i=1}^3 b_i(y) w_i + b_0(y). 
\end{align*} 
On the other hand, expanding the square we have 
\begin{align*}
&\left\vert \grad^g (\rho_0^2)  -  \sum_{i=1}^3 b_i(y) \grad^g w_i \right\vert^2 \\
&\qquad = \vert \grad^g(\rho_0^2)\vert^2 - 2 g(\grad^g (\rho_0)^2, \sum_{i=1}^3 b_i(y)\grad^g w_i) + \sum_{i,j=1}^3 b_i(y)b_j(y) g(\grad^g w_i, \grad^g w_j)\\
&\qquad = 4\rho_0^2 - 2 g(\grad^g (\rho_0)^2, \sum_{i=1}^3 b_i(y)\grad^g w_i) + \sum_{i,j=1}^3 b_i(y)b_j(y) g(\grad^g w_i, \grad^g w_j).
\end{align*} 
Therefore we deduce that 
\[
-2g\left(\grad^g(\rho_0^2), \sum_{i=1}^3 b_i(y)\grad^g w_i\right) + \sum_{i,j=1}^3 b_i(y)b_j(y) g(\grad^g w_i,\grad^g w_j) = -4\sum_{i=1}^3 b_i(y)w_i + b_0(y)
\]
as $L^1_{\text{loc}}$ functions of $x$ for every $y\in \R^3$. We re-write this as 
\begin{equation}
\label{b-coeff-eqn}
\sum_{i,j=1}^3 B^{ij}(x) b_i(y) b_j(y) + \sum_{i=1}^3  \big[4w_i(x) - 2 A_i(x)\big] b_i(y)  = b_0(y)
\end{equation}
where $A_i(x) = g(\grad^g (\rho_0^2), \grad^g w_i)$ and $B^{ij}(x) =g(\grad^g w_i, \grad^g w_j)$. 
Again, for each fixed $y\in \R^3$, this is an equality of $L^1_{\text{loc}}$ functions of $x$. 
Since the variables $(b_1(y),b_2(y),b_3(y))$ range over an open set, this implies that there is a fixed matrix $B^{ij}$ such that $B^{ij}(x) = B^{ij}$ almost-everywhere.    

\begin{prop}
    The matrix $B$ is the identity. 
\end{prop}

\begin{proof} 
The proof of Proposition \ref{harmonic-coordiante} shows that $w_i = x_i + v_i$ where 
\[
v_i = \tilde v_i + u_i.
\]
Here $\grad v_i \in L^{p}_{\delta-1}$ for some $p > 3$ and $0 < \delta < 1$ and $u_i \in \dot H^1$.  By \cite[Theorem 1.2(i)]{bartnik1986mass}, we have $\grad v_i \in L^2_{\beta-1}$ for some $\delta < \beta < 1$ and so 
\[
\int_{\R^3} \vert \grad \tilde v_i\vert^2 (1+\vert x\vert)^{-2\beta-1}\, dx < \infty.
\]
This implies that 
\[
R^{-3} \int_{B_{2R}-B_R} \vert \grad \tilde v_i\vert^2 \,dx \to 0.
\]
Likewise, we have 
\[
\int_{\R^3} \vert \grad u_i\vert^2\, dx < \infty
\]
and so certainly 
\[
R^{-3} \int_{B_{2R}-B_R} \vert \grad u_i\vert^2\, dx < \infty. 
\]
Combining these estimates, we have 
\[
R^{-3} \int_{B_{2R}-B_R} \vert \grad  v_i\vert^2 \,dx \to 0.
\]
With this in place, we now show that $B$ is the identity.

Observe that 
\begin{align*}
  &\frac{1}{\vert B_{2R}-B_R\vert}\int_{B_{2R-B_R}} \vert g(\grad^g w_i,\grad^g w_j) - g_{\text{euc}}(\grad w_i,\grad w_j)\vert \, dx \\
  &\qquad \le C R^{-3-\tau} \int_{B_{2R}-B_R} \vert \grad w_i\vert \vert \grad w_j\vert\, dx\\
  &\qquad \le C R^{-3-\tau} \int_{B_{2R}-B_R} \vert \grad x_i\vert^2 + \vert \grad v_i\vert^2 + \vert \grad x_j\vert^2 + \vert \grad v_j\vert^2\, dx \to 0 
\end{align*}
as $R\to \infty$. In particular, we have 
\[
B^{ij} = \lim_{R\to \infty} \frac{1}{\vert B_{2R}-B_R\vert} \int_{B_{2R}-B_R} \la \grad w_i,\grad w_j\ra\, dx.
\]
We calculate that 
\begin{align*}
    &\frac{1}{\vert B_{2R}-B_R\vert}\int_{B_{2R}-B_R} \la \grad w_i,\grad w_j\ra \, dx \\
    &\qquad = \frac{1}{\vert B_{2R}-B_R\vert}\int_{B_{2R}-B_R} \la \grad x_i,\grad x_j\ra + \la \grad x_i,\grad v_j\ra + \la \grad v_i,\grad x_j\ra + \la \grad v_i,\grad v_j\ra\, dx\\
    &\qquad = \delta_{ij} + \frac{1}{\vert B_{2R}-B_R\vert}\int_{B_{2R}-B_R}  \la \grad x_i,\grad v_j\ra + \la \grad v_i,\grad x_j\ra + \la \grad v_i,\grad v_j\ra\, dx. 
\end{align*}
Now observe that  
\begin{align*}
    \frac{1}{\vert B_{2R}-B_R\vert} \int_{B_{2R}-B_R} \vert \la \grad v_i,\grad v_j\ra\vert\, dx \le C R^{-3} \int_{B_{2R}-B_R} \vert \grad v_i\vert^2 + \vert \grad v_j\vert^2\, dx \to 0
\end{align*}
as $R\to \infty$. Next, since $x_i$ is harmonic in the Euclidean metric, we can integrate by parts to get 
\begin{align*}
    \frac{1}{\vert B_{2R}-B_R\vert}\left\vert \int_{B_{2R}-B_R} \la \grad x_i,\grad v_j\ra\, dx\right\vert &\le CR^{-3} \left[\int_{\bd B_{2R}} \left\vert v_j \frac{\bd x_i}{\bd \nu}\right\vert \, da + \int_{\bd B_R} \left\vert v_j \frac{\bd x_i}{\bd \nu}\right\vert\, da \right] \to 0
\end{align*}
as $R\to \infty$. The term with $\la \grad v_i,\grad x_j\ra$ can be handled similarly. Finally, combining all the estimates, it follows that $B^{ij} = \delta_{ij}$, as needed. 
\end{proof}

Finally we are in position to conclude the rigidity.   

\begin{prop}
Define a map $W\colon \R^3\to \R^3$ by $W(x) = (w_1(x),w_2(x),w_3(x))$.  The map $W$ is an isometry of metric spaces $(\R^3,d_g)\to (\R^3,d_{\operatorname{euc}})$. 
\end{prop}

\begin{proof}
    First we aim to show that $W$ is a homeomorphism. Since $\vert \grad^g w_i\vert^2 = 1$ almost-everywhere, the map $W$ is Lipschitz continuous. The asymptotics for $w_i$ near infinity imply that $W$ is proper and that $W$ has degree one. Thus $W$ must be surjective. 

    Next, we want to check that $W$ is injective. In addition to showing that $B$ is a constant matrix, the above calculations \eqref{b-coeff-eqn} also show that 
    \[
    4w_i(x) - 2g(\grad^g(\rho_0^2),\grad^g w_i) = c_i
    \]
    almost-everywhere for some constants $c_i$. This says that 
    \[
    2 g(\grad^g (\rho_0^2),\grad^g w_i) = 4w_i(x) -c_i
    \]
    and hence 
    \[
    \grad^g (\rho_0^2) = \sum_{i=1}^3 \left(\frac{4 w_i(x)-c_i}{2}\right)\grad^g w_i 
    \]
    almost-everywhere since $\grad^g w_i$ forms a $g$-orthonormal frame. On the other hand, we have $\vert \grad^g (\rho_0^2)\vert^2 = 4\rho_0^2$ and so 
    \[
    4\rho_0^2 = \sum_{i=1}^3 \left(\frac{4w_i(x)-c_i}{2}\right)^2
    \]
    almost-everywhere. Actually, both sides of this equation are continuous and so the equation holds everywhere. Evaluating at the origin and using $\rho_0(0) = 0$, we see that $c_i = 4w_i(0)$ and hence we obtain 
    \[
    \rho_0(x)^2 = \sum_{i=1}^3 (w_i(x) - w_i(0))^2 = \vert W(x) - W(0)\vert^2. 
    \]
    In fact, we could repeat the argument using any other point $y$ as a pole instead of $0$ and hence 
    \[
    \rho_y(x)^2 = \vert W(x) - W(y)\vert^2
    \]
    for all $x$ and $y$. This immediately implies that $W$ is injective since $\rho_y(x) = 0$ only when $x=y$. Finally, since $W$ is continuous, proper, and bijective it follows that $W$ is a homeomorphism. 

    Let $DW$ denote the Jacobian matrix of $W$ and let $J_W = \det(DW)$ be the Jacobian determinant. It is known that a Sobolev homeomorphism has $J_W \ge 0$ almost-everywhere or $J_W \le 0$ almost-everywhere \cite{hencl2010jacobians}. Since $W$ is degree one, we have $J_W \ge 0$ almost-everywhere. Moreover, note that the identity $g(\grad^g w_i,\grad^g w_j) = \delta_{ij}$ implies that
    \[
    (DW)^T(DW) = g
    \]
    almost-everywhere. Now choose $L > 0$ so that the maximum eigenvalue of $g$ at any point is at most $L^2$ and the minimum eigenvalue of $g$ at any point is at least $L^{-2}$. Then  
    it follows that 
    \[
    L^{-1} \vert \xi\vert \le \vert (DW(x))\xi \vert \le L\vert \xi\vert, \quad \xi\in \R^3
    \]
    for almost every $x$. Thus $W$ has $L$-bounded length distortion in the sense of \cite[Section 2]{martio1988elliptic}. Since $W$ is injective, it follows that $W^{-1}$ also has $L$-bounded length distortion \cite[Lemma 2.7]{martio1988elliptic}. In particular, $W^{-1}$ is  Lipschitz and $W^{-1}$ is differentiable with 
    \[
    DW^{-1}(W(x)) = (DW(x))^{-1}
    \]
    almost-everywhere. This implies that 
    \[
(DW^{-1})^T(DW^{-1}) = g^{-1}
    \]
    almost-everywhere.

    Finally, we show that $W$ is an isometry. Consider two points $x,y\in \R^3$. For any $\eps > 0$, we can find a smooth curve $\gamma\colon[0,1]\to \R^3$ connecting $x$ to $y$ such that $W$ is differentiable at almost-every point of $\gamma$ and $L_g(\gamma) \le d_g(x,y) + \eps$. Then we have 
    \begin{align*}
        L_{\text{euc}}(W\circ \gamma) = \int_0^1 \vert (W\circ \gamma)'(t)\vert\, dt &= \int_0^1 \vert DW(\gamma(t))\gamma'(t)\vert\, dt = \int_0^1 \sqrt{g(\gamma'(t),\gamma'(t)}\, dt = L_{g}(\gamma). 
    \end{align*}
    Thus we have $d_{\text{euc}}(x,y) \le L_{\text{euc}}(W\circ \gamma) = L_g(\gamma) \le d_g(x,y) + \eps$ and letting $\eps \to 0$ we deduce that 
    \[
    d_{\text{euc}}(x,y)\le d_g(x,y). 
    \]
    Applying the same argument with $W^{-1}$ in place of $W$ gives the opposite inequality $d_g(x,y) \le d_{\text{euc}}(x,y)$ and thus $W$ is an isometry. 
\end{proof}

\subsection{The Euclidean \texorpdfstring{$C^0$}{C0} Rigidity Conjecture}
Finally, we conclude this section with the application to Gromov's Euclidean $C^0$-Rigidity Conjecture. In light of Theorem \ref{main-theorem} and Theorem \ref{theorem-coordinates}, Corollary \ref{gromov-corollary} is an immediate consequence of the following proposition. 

\begin{prop} Assume that $g$ is a continuous asymptotically flat metric on $\R^3$ which satisfies $\vert g_{ij}(x) - \delta_{ij}\vert = o(\vert x\vert^{-1})$. Then for any function $\eta\in C^\infty_c((1,4))$ we have  
\begin{align*}
    &\lim_{r\to \infty}\bigg[ \frac 1 r \int_{\an(r)} \left(\frac{1}{\vert x\vert} \eta\left(\frac{\vert x\vert}{r}\right) + \frac 1 r \eta'\left(\frac{\vert x\vert}{r}\right) \right)  \delta^{ij}(g_{ij}-\delta_{ij}) \,dx  \\
&\qquad \qquad  + \frac 1 r \int_{\an(r)} \left(\frac{1}{\vert x\vert} \eta\left(\frac{\vert x\vert}{r}\right) - \frac 1 r \eta'\left(\frac{\vert x\vert}{r}\right)\right) (g_{ij}-\delta_{ij}) \frac{x^ix^j}{\vert x\vert^2}\,dx\bigg] = 0.
\end{align*} 
\end{prop}

\begin{proof}
    We can estimate 
    \begin{align*}
    \left\vert\frac 1 r \int_{\an(r)} \left(\frac{1}{\vert x\vert} \eta\left(\frac{\vert x\vert}{r}\right) + \frac 1 r \eta'\left(\frac{\vert x\vert}{r}\right) \right)  \delta^{ij}(g_{ij}-\delta_{ij}) \,dx\right\vert &\le C r  \sup_{\an(r)} \vert \delta^{ij}(g_{ij}-\delta_{ij})\vert \\ &\le C r o(r^{-1})\to 0   
    \end{align*}
    as $r\to \infty$. Likewise, we have 
    \begin{align*}
        \left\vert \frac 1 r \int_{\an(r)} \left(\frac{1}{\vert x\vert} \eta\left(\frac{\vert x\vert}{r}\right) - \frac 1 r \eta'\left(\frac{\vert x\vert}{r}\right)\right) (g_{ij}-\delta_{ij}) \frac{x^ix^j}{\vert x\vert^2}\,dx\right\vert &\le C r \sup_{\an(r)} \left\vert (g_{ij}-\delta_{ij})\frac{x^ix^j}{\vert x\vert^2}\right\vert \\
        &\le C r o(r^{-1})\to 0
    \end{align*}
    as $r\to \infty$. This proves the result. 
\end{proof}

\appendix

\section{The Linearization of \texorpdfstring{$\mathcal D$}{D}}
\label{D-appendix}

Let $\mathcal G(\an)$ be the set of all continuous Riemannian metrics on the compact set $\an = \cl B(0,4) - B(0,1)$.  Define a functional 
\begin{gather*}
\mathcal D\colon \mathcal G(\an) \times W^{1,p}(\an) \to \R,\\
\mathcal D(h,f) = c_\psi + \int_{\an} \varphi(f) {\vert \grad^{h} f\vert^3} \, dv_h,
\end{gather*} 
where $\varphi(f) = \theta(1,f)$ with $\theta$ as in \eqref{eqn-theta}. 
It is straightforward to verify that $\mathcal D$ is jointly continuous in the pair $(h,f)$ provided $p > 3$.   We need to study the linearization of $\mathcal D$.  We first record a simple lemma that will be useful in the proof. 

\begin{lemma}
For any vectors $X,Y\in \R^3$ one has 
\[
\big\vert \vert X+Y\vert^3 - \vert X\vert^3 - 3 \vert X\vert \la X,Y\ra\big \vert \le C (\vert X\vert \vert Y\vert^2 + \vert Y\vert^3) 
\]
where $C$ does not depend on $X$ and $Y$. 
\end{lemma}

\begin{proof}
We calculate that 
\begin{align*}
\vert X + Y\vert^3 - \vert X\vert^3 &= \int_0^1 \frac{d}{dt} \la X+tY,X+tY\ra^{3/2} \, dt \\
&= 3 \int_0^1 \la X+tY,X+tY\ra^{1/2} \la X+tY,Y\ra\, dt.
\end{align*}
Then we estimate 
\begin{align*}
\int_0^1\big\vert \vert X+tY\vert \la X+tY,Y\ra - \vert X\vert \la X,Y\ra \big\vert\, dt &\le \int_0^1 \big\vert \vert X+tY\vert - \vert X\vert\big\vert \vert X\vert \vert Y\vert + t \vert X+tY\vert \vert Y\vert^2\, dt \\
&\le C \int_0^1 \vert X\vert \vert Y\vert^2 + \vert Y\vert^3\, dt\\
&= C(\vert X\vert \vert Y\vert^2 + \vert Y\vert^3). \phantom{\int} 
\end{align*}
The result follows. 
\end{proof} 

Now we proceed to show the differentiability of $\mathcal D$.   

\begin{prop}
\label{D-Frechet}
The map $\mathcal D$ is Frechet differentiable at the point $(\delta, \rho) = (g_{\operatorname{euc}},\vert y\vert^{-1})$ with 
\begin{align*}
\mathcal D'_{(\delta,\rho)} (k,v) = L(k,v) &:= \int_{\an} \varphi'(\rho)  \vert \grad \rho\vert^3 v\, dy \\
&\qquad + 3 \int_{\an} \varphi(\rho) \vert \grad \rho\vert \la \grad \rho, \grad v\ra \, dy\\
&\qquad + \int_{\an} \varphi(\rho) \left[\frac 1 2 (\tr k) \vert \grad \rho\vert^3 - \frac 3 2 \vert \grad \rho\vert k(\grad \rho,\grad \rho)\right]\, dy. 
\end{align*}
In fact, we have 
\[
\mathcal D(\delta+k,\rho+v) = \mathcal D(\delta,\rho) + L(k,v) + O(\|k\|_{C^0}^2 + \|k\|_{C^0}\|v\|_{W^{1,p}} + \|v\|_{W^{1,p}}^2)
\]
with a quadratic error estimate. 
\end{prop}

\begin{proof}
Note that any symmetric 2-tensor sufficiently close to $\delta$ in $C^0$ will be uniformly positive definite and that any function $f$ sufficiently close to $\rho$ in $W^{1,p}$ will satisfy $f \ge c > 0$ for a constant $c$ that does not depend on $f$.  
We proceed in several steps. 

{\bf Step 1.} First, we freeze the metric at the Euclidean metric and prove the Frechet differentiability of the functional 
\[
J(f) = \mathcal D(\delta,f) = c_\psi + \int_{\an} \varphi(f) \vert \grad f\vert^3\, dy. 
\]
at $f = \rho$.  We claim that 
\[
J'_{\rho}(v) = L_1(v) := \int_{\an} \varphi'(\rho)  \vert \grad \rho\vert^3 v\, dy 
+ 3 \int_{\an} \varphi(\rho) \vert \grad \rho\vert \la \grad \rho, \grad v\ra \, dy.
\]
We are going to show that 
\[
J(\rho+v) - J(\rho) - L_1(v) = O(\|v\|_{W^{1,p}}^2). 
\]
To this end, we observe that there is a uniform estimate 
\[
\varphi(\rho + v) = \varphi(\rho) + \varphi'(\rho)v + O(\|v\|_{C^0}^2)
\]
once $\|v\|_{C^0}$ is sufficiently small. Also, by the previous lemma, there is a uniform estimate 
\[
\vert \grad \rho + \grad v\vert^3 = \vert \grad \rho\vert^3 + 3 \vert \grad \rho\vert \la \grad \rho, \grad v\ra + O(\vert \grad \rho\vert \vert \grad v\vert^2 + \vert \grad v\vert^3). 
\]
It follows that 
\begin{align*}
J(\rho+v)-J(\rho) &= \int_{\an} \varphi'(\rho) \vert \grad \rho\vert^3 v + 3\varphi(\rho)\vert \grad \rho\vert \la \grad \rho,\grad v\ra\, dy \\
&\qquad + \int_{\an} 3\varphi'(\rho)\vert \grad \rho\vert v \la \grad \rho,\grad v\ra + (\vert \grad \rho\vert^3 + 3\vert \grad \rho\vert \la \grad \rho,\grad v\ra)O(\|v\|_{C^0}^2)\, dy \\
&\qquad + \int_{\an} (\varphi(\rho) + \varphi'(\rho)v + O(\|v\|_{C^0}^2)) O(\vert \grad \rho\vert \vert \grad v\vert^2 + \vert \grad v\vert^3)\, dy. 
\end{align*}
It is easy to see that the remainder terms in the second and third lines are in fact uniformly $O(\|v\|_{W^{1,p}}^2)$. This proves the differentiability of $J$. 

{\bf Step 2.} Next, define 
\[
L_2(k;f) =  \int_{\an} \varphi(f) \left[\frac 1 2 (\tr k) \vert \grad f\vert^3 - \frac 3 2 \vert \grad f\vert k(\grad f,\grad f)\right]\, dy. 
\]
We claim that 
\[
\mathcal D(\delta + k, \rho+v) - \mathcal D(\delta,\rho+v) - L_2(k,\rho + v) = O(\|k\|_{C^0}^2). 
\]
Define the function 
\[
\Psi(H,\xi) = \la H^{-1}\xi,\xi\ra^{3/2} \sqrt{\det H}
\]
where $H$ is a positive definite symmetric matrix and $\xi\in S^2$.  This is a smooth function of $(H,\xi)$ for $H$ in a neighborhood of the identity matrix and $\xi \in S^2$.  By Taylor expansion, for a symmetric matrix $K$ close to 0, we have  
\[
\Psi(I+K,\xi) = \Psi(I,\xi) + \Psi'_{(I,\xi)}(K,0) + O(\|K\|^2)
\]
where the constant in $O(\|K\|^2)$ does not depend on $\xi$. Now since $\Psi$ is homogeneous in $\xi$, we obtain 
\[
\Psi(I+K,\xi) = \Psi(I,\xi) + \Psi'_{(I,\xi)}(K,0) + \vert \xi\vert^3 O(\|K\|^2)
\]
for any $\xi \neq 0$ where again the constant in $O(\|K\|^2)$ doesn't depend on $\xi$. Finally, we note that this equation is also valid for $\xi = 0$ since then every term is equal to $0$.  

We compute that 
\[
\Psi'_{(I,\xi)}(K,0) =  \frac 1 2 \vert \xi\vert^3 \tr(K) - \frac 3 2 \vert \xi\vert \la K\xi,\xi\ra.
\]
Now observe that 
\begin{align*}
&\mathcal D(\delta+k,\rho+v)-\mathcal D(\delta,\rho+v) \\
&\qquad = \int_{\an} \varphi(\rho+v) [\Psi(\delta+k,\grad \rho + \grad v) - \Psi(\delta,\grad \rho + \grad v)]\, dy\\
&\qquad = \int_{\an} \varphi(\rho+v) \left[ \frac 1 2 \vert \grad \rho + \grad v\vert^3 \tr(k) - \frac 3 2 \vert \grad \rho + \grad v\vert  k(\grad \rho + \grad v,\grad \rho + \grad v)\right]\, dy  \\
&\qquad \qquad \qquad   + \int_{\an} \varphi(\rho+v)\left[ \vert \grad \rho + \grad v\vert^3 O(\|k\|_{C^0}^2)\right]\, dy\\
& \qquad = L_2(k;\rho+v) + C \|k\|_{C^0}^2
\end{align*}
where we used the fact that 
\[
\int_{\an} \vert \grad \rho + \grad v\vert^3 \, dy \le \int_{\an} \vert \grad \rho\vert^3 + 3 \vert \grad \rho\vert^2 \vert \grad v\vert + 3 \vert \grad \rho\vert \vert \grad v\vert^2 + \vert \grad v\vert^3\, dy \le C. 
\]
This completes Step 2.  

{\bf Step 3:} Combining Steps 1 and 2, we deduce that 
\[
\mathcal D(\delta+k,\rho+v)-\mathcal D(\delta,\rho) = L_1(v) + L_2(k;\rho+v) + O(\|v\|_{W^{1,p}}^2 + \|k\|_{C^0}^2). 
\]
To complete the proof we need to show that 
\[
L_2(k;\rho+v) = L_2(k,\rho) + O(\|k\|_{C^0}^2 + \|k\|_{C^0}\|v\|_{W^{1,p}} + \|v\|_{W^{1,p}}^2).
\] 
We now verify this estimate.  
We compute that 
\begin{align*}
&\vert L_2(k,\rho+v) - L_2(k,\rho)\vert\\
&\qquad \le C \|k\|_{C^0} \int_{\an} \left\vert \varphi(\rho+v) \vert \grad \rho + \grad v\vert^3 - \varphi(\rho)\vert \grad \rho\vert^3\right\vert \,dy \\
&\qquad \quad + C \int_{\an} \left\vert \varphi(\rho+v) \vert \grad \rho + \grad v\vert  k(\grad \rho + \grad v,\grad \rho+\grad v) - \varphi(\rho)\vert \grad \rho\vert k(\grad \rho,\grad \rho)\right\vert\, dy.
\end{align*} 
Let us estimate the two integrals separately. For the first, observe that 
\begin{align*}
&\left \vert \varphi(\rho+v) \vert \grad \rho + \grad v\vert^3 - \varphi(\rho)\vert \grad \rho\vert^3\right\vert \\
&\qquad \le \vert \varphi(\rho+v) - \varphi(\rho)\vert \vert \grad \rho + \grad v\vert^3 + \vert \varphi(\rho)\vert \left\vert \vert \grad \rho+\grad v\vert^3 - \vert \grad \rho\vert^3\right\vert\\
&\qquad \le C \|v\|_{C^0}(1 + \vert \grad v\vert + \vert \grad v\vert^2 + \vert \grad v\vert^3) + C (1+\vert \grad v\vert)^2\vert \grad v\vert
\end{align*}
which yields 
\[
C \|k\|_{C^0}  \int_{\an} \left\vert \varphi(\rho+v) \vert \grad \rho + \grad v\vert^3 - \varphi(\rho)\vert \grad \rho\vert^3\right\vert \,dy = O(\|k\|_{C^0}^2 + \|k\|_{C^0}\|v\|_{W^{1,p}} + \|v\|_{W^{1,p}}^2)
\]
upon integrating. Now consider the second integral.  Again we use the triangle inequality to get 
\begin{align*}
&\left\vert \varphi(\rho+v) \vert \grad \rho + \grad v\vert  k(\grad \rho + \grad v,\grad \rho+\grad v) - \varphi(\rho)\vert \grad \rho\vert k(\grad \rho,\grad \rho)\right\vert\\
&\qquad \le \vert \varphi(\rho+v)-\varphi(\rho)\vert \vert \grad \rho + \grad v\vert \vert k(\grad \rho + \grad v,\grad \rho+\grad v) \vert \\
&\qquad \qquad + \vert \varphi(\rho)\vert \vert \grad \rho + \grad v\vert - \vert \grad \rho\vert\vert \vert k(\grad \rho + \grad v,\grad \rho+\grad v)\vert \\
&\qquad \qquad + \vert \varphi(\rho)\vert \vert \grad \rho\vert \vert k(\grad \rho+\grad v,\grad \rho+\grad v)-k(\grad \rho,\grad \rho)\vert. 
\end{align*} 
The term on the second line is bounded by 
\[
C \|v\|_{C^0} \|k\|_{C^0} (1+\vert \grad v\vert)^3,
\]
and the term on the third line is bounded by 
\[
C \vert \grad v\vert \|k\|_{C^0} (1+\vert \grad v\vert)^2. 
\]
Finally, the term on the fourth line is bounded by 
\[
C \vert 2k(\grad \rho,\grad v) + k(\grad v,\grad v)\vert \le C \|k\|_{C^0} \vert \grad v\vert + C\|k\|_{C^0} \vert \grad v\vert^2. 
\]
Combining these estimates and integrating shows that 
\begin{align*}
&\int_{\an} \left\vert \varphi(\rho+v) \vert \grad \rho + \grad v\vert  k(\grad \rho + \grad v,\grad \rho+\grad v) - \varphi(\rho)\vert \grad \rho\vert k(\grad \rho,\grad \rho)\right\vert\, dy\\
&\qquad  = O(\|k\|_{C^0}^2 + \|k\|_{C^0}\|v\|_{W^{1,p}} + \|v\|_{W^{1,p}}^2). 
\end{align*} 
Thus we have $L_2(k;\rho+v) = L_2(k;\rho) + O(\|k\|_{C^0}^2 + \|k\|_{C^0}\|v\|_{W^{1,p}} + \|v\|_{W^{1,p}}^2)$, as desired.  Hence we have now demonstrated that $\mathcal D$ is differentiable with derivative $\mathcal D'_{(\delta,\rho)}(k,v) = L_1(v) + L_2(k;\rho)$ and a quadratic error estimate. 
\end{proof}

\section{The \texorpdfstring{$F$}{F} Function}

\label{Appendix-F}

In this appendix, we summarize some results about the $F$ function introduced by Agostiniani, Mazzieri, and Oronzio in \cite{agostiniani2024green}. Let $g$ be a smooth metric on $\R^3$ and let $u$ be a positive harmonic function on some region with closed, connected level sets.  Further assume that $u$ is scaled so that the flux over each level set of $u$ is $4\pi$.   

\begin{defn}
For regular values $t > 0$ of $u$, the $F$ function is defined by 
\[
F(t) = 4\pi t - t^2 \int_{\{u=\frac 1 t\}} H \vert \grad u\vert\, da + t^3 \int_{\{u=\frac 1 t\}} \vert \grad u\vert^2\, da,
\]
where all the geometric quantities are computed with respect to the $g$ metric. 
\end{defn}

\begin{rem}
The function $u$ above corresponds to the function $1 - u$ in the original notation of \cite{agostiniani2024green}. We hope this will not cause confusion. Also, our convention is that 
\[
H = -\div\left(\frac{\grad u}{\vert \grad u\vert}\right)
\]
so that $H$ is positive when $u = \vert x\vert^{-1}$ on Euclidean space. 
\end{rem} 

One has the following monotonicity theorem for the $F$ function \cite{agostiniani2024green}.  

\begin{prop}
\label{F-derivative}
The function $F$ is absolutely continuous with 
\[
F'(t) = 4\pi + \int_{\{u=\frac 1 t\}} -\frac{R^{\Sigma_t}}{2} + \frac{\vert \grad^{\Sigma_t} \vert \grad u\vert \vert^2}{\vert \grad u\vert^2} + \frac{R}{2} + \frac{\vert \mathring A\vert^2}{2} + \frac{3}{4}\left(\frac{2\vert \grad u\vert}{u}-H\right)^2\, da
\]
for almost every $t$. Here we have abbreviated $\Sigma_t = \{u = \frac 1 t\}$. 
\end{prop} 

This has the following corollaries. 

\begin{corollary}
Assume that $g$ has non-negative scalar curvature.  Then $F(t)$ is a non-decreasing function of $t$. 
\end{corollary}

\begin{proof}
This follows from the formula for $F'(t)$ together with the Gauss-Bonnet theorem since the level sets of $u$ are connected. 
\end{proof}

Finally, we note that the term involving the gradient squared is absolutely continuous on its own. 

\begin{prop}
\label{ac-gradient}
    The function $t\mapsto \int_{\{u=\frac 1 t\}} \vert \grad u\vert^2\, da$ is absolutely continuous with 
    \[
\frac{d}{dt} \int_{\{u=\frac 1 t\}} \vert \grad u\vert^2\, da = -t^{-2}\int_{\{u=\frac 1 t\}} H \vert \grad u\vert\, da
\]
    almost everywhere. 
\end{prop}

\section{Estimates for the Laplace Equation} 
\label{Appendix-Mapping}

The goal of this appendix is to prove estimates for the weak solution of $\lap w = \div(X)$ on $\R^3$ in various weighted spaces. The following discussion is largely based on \cite{bartnik1986mass} and  \cite{farwig1997weighted}. Define the weight 
\[
\sigma(x) = 1+\vert x\vert. 
\]
Following \cite{bartnik1986mass}, for a real number $\delta$, define the weighted $L^q$ and Sobolev norms 
\[
\|f\|_{L^q_{\delta}}^q = \int_{\R^3} \vert f\vert^q \sigma^{-q\delta-3}\, dx,\quad 
\|f\|_{W^{1,q}_\delta}^q = \|f\|_{L^q_{\delta}}^q + \|\grad f\|_{L^q_{\delta-1}}^q.
\]
Define the homogeneous Sobolev space 
    \[
    \dot W^{1,q}_\delta(\R^3) = \{u\in W^{1,1}_{\text{loc}}(\R^3): \|\grad u\|_{L^q_{\delta-1}} < \infty\}. 
    \]
    This becomes a separable Banach space after taking a quotient by the constant functions. 
    For a given $1 < q < \infty$, the function $
\sigma(x)^{-q(\delta-1) - 3}
$ is a Muckenhoupt $A_q$ weight for 
\[
-3 < -q(\delta-1) -3 < 3(q-1);
\] 
see \cite[Page 252]{farwig1997weighted}.
This is equivalent to 
\[
-2 < \delta < 1.
\]
We are going to need the following mapping properties of the Laplacian. 

\begin{rem}
    The restriction to the range $-2<\delta<1$ in the following propositions can intuitively be understood by the fact that the linear harmonic functions $x_i$ belong to $W^{1,q}_\delta$ for $\delta > 1$ and to the dual of $W^{1,q}_\delta$ for $\delta < -2$. 
\end{rem}

\begin{prop}
\label{weight-estimate-I}
    Fix some $-2 < \delta< 0$ and some $1 < q < \infty$. Assume that $X$ is a vector field in $L^q_{\delta-1}(\R^3)$.  Then there is a unique weak solution to \[
    \lap w = \div X, \quad w\in W^{1,q}_{\operatorname{loc}}(\R^3)
    \]
satisfying the estimate 
    \[
    \|w\|_{W^{1,q}_\delta} \le C \|X\|_{L^q_{\delta-1}}
    \]
    where $C$ does not depend on $X$. 
\end{prop}

\begin{proof}
    According to \cite[Theorem 4.2(i)]{farwig1997weighted}, there is a weak solution $w\in \dot W^{1,q}_\delta(\R^3)$ to $\lap w = \div(X)$. Moreover, $w$ is unique up to an additive constant $c$ and 
    \[
    \|\grad w\|_{L^q_{\delta-1}} \le C \|X\|_{L^q_{\delta-1}}. 
    \]
    Since $-2 < \delta < 0$, the weight $\sigma^{-q\delta - 3}$ belongs to $A_q$. Hence we can apply the Poincare-type inequality \cite[Corollary 3.7(i)]{farwig1997weighted} with $r = q$ to deduce that there is a unique choice of the constant $c$ for which $w\in L^q_\delta(\R^3)$ with an estimate 
    \[
    \|w\|_{L^q_\delta} \le C \|\grad w\|_{L^q_{\delta-1}}.
    \]
    This $w$ is as required. 
\end{proof}

To handle the case of $0 < \delta < 1$, we need to use Bartnik's Poincare inequality. 

\begin{prop}
Fix some $0 < \delta < 1$. Then for any $w$ we have an estimate 
\[
\inf_{c\in \R} \|w - c\|_{L^q_\delta} \le C \|\grad w\|_{L^q_{\delta-1}}.
\]
In fact, the estimate holds for $c = \frac{1}{\vert B_2\vert}\int_{B_2} w\, dx$.
\end{prop}

\begin{proof}
    Choose a smooth cutoff function $\eta$ such that $\eta\equiv 1$ on $B(0,1)$ and $\eta\equiv 0$ outside $B(0,2)$. Define 
    \[
    c = \frac{1}{\vert B_2\vert}\int_{B_2} w\, dx. 
    \]
    We can apply Bartnik's Poincare inequality \cite[Theorem 1.3(i)]{bartnik1986mass} to deduce that 
    \begin{align*} 
    &\int_{\R^3-B_1} \vert(1-\eta)(w-c)\vert^q \vert x\vert^{-q\delta-3}\, dx \\
    &\qquad \le \delta^{-1} \int_{\R^3-B_1} \vert \grad ((1-\eta)(w-c))\vert^q \vert x\vert^{-q(\delta-1)-3}\, dx. 
    \end{align*}
    This implies that
    \begin{align*} 
    &\int_{\R^3-B_1} \vert(1-\eta)(w-c)\vert^q \sigma^{-q\delta-3}\, dx &\\&\qquad \le C \int_{\R^3} \vert \grad ((1-\eta) (w-c))\vert^q \sigma^{-q(\delta-1)-3}\, dx\\
    &\qquad \le C \int_{\R^3} \vert \grad w\vert^q \sigma^{-q(\delta-1)-3}\, dx + C \int_{B_2} \vert w-c\vert^q \,dx.  
    \end{align*}
    for a constant $C >0$ that does not depend on $w$. Hence by the ordinary Poincare inequality on $B_2$, we deduce that 
    \[
    \int_{\R^3 - B_2} \vert w-c\vert^q \sigma^{-q\delta-3}\, dx \le C \|\grad w\|_{L^q_{\delta-1}(\R^3)}^q + C \|\grad w\|_{L^q(B_2)}^q \le C \|\grad w\|_{L^q_{\delta-1}(\R^3)}^q. 
    \]
    By the ordinary Poincare inequality again, we have 
    \[
    \int_{B_2} \vert w-c\vert^q \sigma^{-q\delta-3}\, dx \le C \int_{B_2} \vert w-c\vert^q\, dx \le C \|\grad w\|_{L^q(B_2)}^q \le C \|\grad w\|_{L^q_{\delta-1}(\R^3)}^q.
    \]
    Adding the two estimates, it follows that 
    \[
    \|w - c\|_{L^q_\delta} \le C \|\grad w\|_{L^q_{\delta-1}}
    \]
    for $c = \frac{1}{\vert B_2\vert}\int_{B_2} w \, dx$. This proves the proposition. 
\end{proof}

This implies the analogous existence result in the range $0 < \delta < 1$. 

\begin{prop}
\label{weight-estimate-II}
Fix some $0 < \delta < 1$ and some $1 < q < \infty$. Assume that $X$ is a vector field in $L^q_{\delta-1}(\R^3)$. Then there is a unique weak solution to 
\[
\lap w = \div X,\quad  w\in W^{1,q}_{\operatorname{loc}}(\R^3) 
\]
satisfying the normalization
\[
\int_{B_2} w\, dx =0. 
\]
Moreover, $w$ satisfies the estimate 
\[
\|w\|_{W^{1,q}_\delta} \le C \|X\|_{L^q_{\delta-1}}
\]
where $C$ does not depend on $X$. 
\end{prop}

\begin{proof}
     By  \cite[Theorem 4.2(i)]{farwig1997weighted} again, there is a weak solution $w\in \dot W^{1,q}_\delta(\R^3)$ to $\lap w = \div(X)$ satisfying the estimate 
    \[
    \|\grad w\|_{L^q_{\delta-1}}\le C \|X\|_{L^q_{\delta-1}}.
    \]
    Moreover, $w$ is unique up to an additive constant $c$.  There is a unique choice of $c$ for which $w$ satisfies 
    \[
    \int_{B_2} w\, dx = 0
    \]
    and for this choice of $c$ we have 
    \[
    \|w\|_{L^q_\delta} \le C \|\grad w\|_{L^q_{\delta-1}}
    \]
    by the previous proposition.  Combining these estimates, the result follows. 
\end{proof}

\section{Relation with Burkhardt-Guim's Global Mass}
\label{Appendix-eta}

In this appendix, we investigate the relationship between the harmonic mass and Burkhardt-Guim's global $C^0$ mass. As a first step, we determine the precise form of the function $\eta = \eta_\psi$ in Theorem \ref{theorem-coordinates2}. 

\begin{prop}
    Let $\psi\not\equiv 0$ be a smooth non-negative bump function compactly supported in $(0,1)$. Then the function $\eta = \eta_\psi$ in Theorem \ref{theorem-coordinates2} has the explicit form 
    \[
    \eta(s) = \frac{1}{2s^3}\int_{\frac s 2 - 1}^{s-1} \psi(t)\, dt.
    \]
    In particular, $\eta$ is non-negative. 
\end{prop}

\begin{proof}
    Inspecting the proof of Theorem \ref{theorem-coordinates2}, we have 
    \[
    \eta(s) = -\frac 1 2 s^{-5} \varphi(s^{-1}) + \frac 1 4 s^{-1}\chi'(s)
    \]
    where 
    \begin{gather*}
        \varphi(s) = s^{-3}\left[\frac 1 2 \psi\left(\frac{1}{2s}-1\right)-\psi\left(\frac 1 s -1\right)\right],\\
        \chi'(s) = s^{-2}\left[\int_0^s t^2 \Theta(t)\, dt\right],\\
        \Theta(s) = -6s^{-5}\varphi\left(\frac 1 s\right) - 2s^{-6}\varphi'\left(\frac 1 s\right) . 
    \end{gather*}
    It is convenient to introduce another function 
    \[
    f(t) = \frac 1 2 \psi\left(\frac t 2 - 1\right) - \psi\left(t-1\right)
    \]
    so that 
    \[
\varphi\left(\frac 1 s\right) = s^3 f(s). 
    \]
    We compute that 
    \[
    \varphi'\left(\frac 1 s\right)=-s^2 \frac{d}{ds}\left[\varphi\left(\frac 1 s\right)\right] = -s^2\frac{d}{ds}[s^3f(s)] = -3s^4 f(s)- s^5f'(s).
    \]
    Therefore we have 
    \begin{align*}
    \Theta(s) &= -6s^{-2}f(s) -2s^{-6}(-3s^4f(s)-s^5f'(s))= 2 s^{-1} f'(s). 
\end{align*}
We then calculate that 
\[
\chi'(s) = 2s^{-2}\int_0^s tf'(t)\, dt = 2s^{-2}\left[ sf(s) - \int_0^sf(t)\,dt\right].
\]
Making these substitutions into the formula for $\eta$, we see that 
\[
\eta(s) = -\frac 1 2 s^{-3}\int_0^s f(t)\, dt. 
\]
Finally, by suitable changes of variables, we have 
\[
\int_0^s f(t)\, dt = \int_0^s \frac 1 2 \psi\left(\frac t 2 -1\right) -\psi(t-1)\, dt = -\int^{ s -1}_{\frac s 2-1}\psi(t)\, dt
\]
Thus we have 
\[
\eta(s) = \frac{1}{2s^3}\int_{\frac s 2 - 1}^{s-1} \psi(t)\, dt,
\]
as claimed. Since $\psi$ is non-negative, it immediately follows that $\eta$ is non-negative.  
\end{proof}

Next, we recall the definition of Burkhardt-Guim's family $\eta_{\theta}$. We will continue to work with our convention of using the interval $(1,4)$ instead of the interval $(0.9,1.1)$. By \cite[Lemma 4.1]{burkhardt2024adm}, for small $0 < \theta < \theta_0$ there is a function $\eta_\theta\colon (0,\infty) \times [0,\infty) \to \R$ satisfying 
\begin{equation}
\label{eta-ode}
\begin{cases}
    \bd_t \eta_\theta(\ell,t) = -\bd_{\ell \ell} \eta_\theta(\ell,t) - \frac 2 \ell \bd_\ell \eta_\theta(\ell,t)+ \frac 2 {\ell^2} \eta_\theta(\ell,t),\\
    \eta_\theta(\ell,\theta) = \eta(\ell). 
\end{cases}
\end{equation}
Fix some $0 < \alpha < 2\tau - 1$ and define $\eta^r(\ell) = \eta_{r^{-\alpha}}(\ell,0)$. Then Burkhardt-Guim's global mass is defined by
\[
M_{C^0}(g) = \lim_{r\to \infty} M_{C^0}(g,\text{Id},\eta^r,r). 
\]
We want to show that for fast decay $\tau > \frac 2 3$, we have 
\[
\lim_{r\to \infty} M_{C^0}(g,\text{Id},\eta^r,r) = \lim_{r\to \infty} M_{C^0}(g,\text{Id},\eta,r).
\]
Then by Theorem \ref{theorem-coordinates}, it follows that $M_{C^0}(g)$ coincides with the harmonic mass $m(g)$.  Since non-negative scalar curvature in the Ricci flow sense implies non-negative scalar curvature in the sense of approximations \cite[Lemma 7.2]{burkhardt2019pointwise}, this answers \cite[Question 2]{burkhardt2024adm} for metrics on $\R^3$ with fast decay $\tau > \frac 2 3$. 

\begin{prop} Assume that $\tau > \frac 2 3$. Then 
\[
\lim_{r\to \infty} M_{C^0}(g,\operatorname{Id},\eta^r,r) = \lim_{r\to \infty} M_{C^0}(g,\operatorname{Id},\eta,r).
\]
\end{prop}

\begin{proof}
    Let $\an(r) = B(0,4r)-B(0,r)$. From the definition of Burkhardt-Guim's $C^0$ local mass \cite[Definition 2.1]{burkhardt2024adm}, it is easy to check that 
    \begin{align*}
        &\vert M_{C^0}(g,\text{Id},\eta^r,r)-M_{C^0}(g,\text{Id},\eta,r)\vert \le  C r \|\eta^r-\eta\|_{C^1([1,4])}  \|g_{ij}-\delta_{ij}\|_{L^\infty(\an(r))}.
    \end{align*}
    Hence, to prove the proposition, it suffices to show that $r^{1-\tau}\|\eta^r-\eta\|_{C^1([1,4])}\to 0$ as $r\to \infty$. It remains to estimate $\|\eta^r - \eta\|_{C^1([1,4])}$. 

    For convenience, define $f_\theta(\ell,t) = \eta_\theta(\ell,\theta-t)$.  
In fact, $f_\theta$ is then independent of $\theta$ and we denote it simply by $f$. It solves the forward parabolic equation 
\[
\begin{cases}
    \bd_t f(\ell,t) = \bd_{\ell \ell} f(\ell,t) + \frac 2 \ell \bd_\ell f(\ell,t)- \frac 2 {\ell^2} f(\ell,t),\\
    f(\ell,0) = \eta(\ell).
\end{cases}
    \]
    In fact, from the proof of \cite[Lemma 4.1]{burkhardt2024adm}, we see that $f(\ell,t) = \bd_\ell \tilde u(\ell,t)$ where $\tilde u(\vert x\vert,t) = u(x,t)$ and $u$ solves the heat equation on $\R^3$ with smooth radially symmetric initial data which is constant outside of a compact set. In particular, $f$ is smooth on $(0,\infty)\times [0,\theta_0]$. It now follows easily from the smoothness that there is a constant $C$ such that 
    \[
    \|f(\cdot,t)-f(\cdot,0)\|_{C^1([1,4])} \le Ct. 
    \]
    Thus we have 
    \[
    \|\eta_\theta(\cdot,0)-\eta(\cdot)\|_{C^1([1,4])} \le C\theta, 
    \]
    and it follows that 
    \[
    \|\eta^r - \eta\|_{C^1([1,4])} = \|\eta_{r^{-\alpha}}(\cdot,0)-\eta(\cdot)\|_{C^1([1,4])} \le C r^{-\alpha}. 
    \]
    Hence we obtain 
    \[
    r^{1-\tau} \|\eta^r-\eta\|_{C^1([1,4])} \le C r^{1-\tau - \alpha}. 
    \]
    Choosing $\alpha$ very close to $2\tau - 1$ and using $\tau > \frac 2 3$, we can arrange that $1 - \tau - \alpha < 0$ and hence this goes to $0$ as $r\to \infty$.
    \end{proof} 

    As explained above, when combined with the results of this paper, this has the following corollary. 

    \begin{corollary}
        Assume $g$ is a continuous asymptotically flat metric on $\R^3$ of decay rate $\tau > \frac 2 3$. If $g$ has non-negative scalar curvature in the sense of Ricci flow then 
        \[
        \lim_{r\to\infty} M_{C^0}(g,\operatorname{Id},\eta^r,r) \ge 0. 
        \]
    \end{corollary}

\bibliographystyle{plain}
\bibliography{bibliography}

\end{document}